\documentclass{amsart} 
\usepackage{a4wide} 

\usepackage{amsthm}
\usepackage{amssymb}
\usepackage{amsmath} 
\usepackage{amscd}
\usepackage{color}

\newtheorem{theorem}{Theorem}[section]
\newtheorem{lemma}[theorem]{Lemma}
\newtheorem{corollary}[theorem]{Corollary}
\newtheorem{proposition}[theorem]{Proposition}
\newtheorem{example}[theorem]{Example}

\theoremstyle{definition}

\newcommand{\Hom}{\mbox{Hom}}
\newcommand{\Mor}{\mbox{Mor}}

\newcommand{\Set}{\mbox{Set}}

\newcommand{\Ac}{\mathcal{A}}

\newcommand{\Bc}{\mathcal{B}}

\newcommand{\Z}{\mathbb{Z}}

\newcommand{\Cc}{\mathcal{C}}

\newcommand{\M}{\mathcal{M}}

\DeclareMathOperator{\id}{\rm id}

\DeclareMathOperator{\A}{\rm Act}
\DeclareMathOperator{\oA}{{\rm Act}}

\begin{document}
\title{Connected objects in categories of $S$-acts}

\author{Josef Dvo\v r\'ak}
\email{pepa.dvorak@post.cz}
\address{CTU in Prague, FEE, Department of mathematics, Technick\'a 2, 166 27 Prague 6 \&
	MFF UK, Department of Algebra,  Sokolovsk\' a 83, 186 75 Praha 8, Czechia}

\author{Jan \v Zemli\v cka}
\email{zemlicka@karlin.mff.cuni.cz}
\address{Department of Algebra, Charles University in Prague,
Faculty of Mathematics and Physics Sokolovsk\' a 83, 186 75 Praha 8, Czechia}

\begin{abstract} 
The main goal of the paper is description of connected and projective objects of classes of categories that include categories of acts along with categories of pointed acts.
In order to establish a general context and to unify the approach to both of the categories of acts, the notion of a concrete category with a unique decomposition of objects is introduced and studied. Although these categories are not extensive in general,
it is proved in the paper that they satisfy a version of extensivity, which ensures that every noninitial object is uniquely decomposable into indecomposable objects. 
\end{abstract}

\subjclass[2010]{20M50  (20M30)}

\keywords{connected object, act over a monoid, category of acts}

\thanks{This work is part of the project SVV-2020-260589}
\date{\today}

\maketitle

\section{Introduction}

While the great impact of category theory on theory of rings and modules is well known, the analogous concept in the context of theory of monoids and acts over monoids on sets is significantly less studied, however it seems to be promising and fruitful as it is demonstrated in the monograph \cite{KKM}. The aim of this paper
is to introduce and study a class of categories including the most important categories of acts, which are given by acting of a general monoid on sets 
and by acting of a monoid with zero on pointed sets. However a category of pointed acts over a monoid with zero seems to be closer to a module category then a category of standard acts, our main goal  is to develop tools that can apply simultaneously to both cases.

Recall that an object $c$ of an abelian
category closed under coproducts and products is said to be {\it compact} if the corresponding covariant functor $\Hom(c,-)$ commutes with arbitrary coproducts i.e. there is a canonical isomorphism in the category of abelian groups $\Hom(c,\coprod \mathcal D)\cong \coprod\Hom(c,\mathcal D)$ for every family of objects $\mathcal D$, where $\coprod$ denotes a coproduct \cite{DZ21, Tr}. 
Compact objects are called {\it small} in categories of (right $R$-)modules \cite{AM,CT,Dv}. The notion of a compact object in non-abelian categories  has different meaning; it is usually defined as an object such that the corresponding covariant hom-functors commute with filtered colimits, nevertheless, it can be proved that that this notion is stronger \cite{EKN03}. In compliance with convention, an object $c$ such that  $\Mor(c,-)$ commutes with all coproducts is called {\it connected} in this paper.

The main motivation of the present paper, reflected by its title, is an issue of translating description of connectedness from abelian categories to a more general context. The constitutive example of such a generalization is provided by the analogy between (abelian) categories of modules over rings
and (non-abelian) categories of acts over monoids (cf. also the corresponding description of connectedness in the case of Ab5 categories \cite{DZ21}). Moreover, note that both the categories of acts and pointed acts are quasi pointed categories (see \cite{Bo}).

The list of works dedicated to the research of connected and compact objects in various categories is long. Let us mention only those related to our concept of linking (self)small modules, (auto)compact objects in abelian categories and (auto)connected objects in non-abelian context.
However, the notion of autocompactness of modules \cite[Proposition 1.1]{AM} was generalized to Grothendieck 
categories in \cite{GNM}, the main motivation for the study of connected objects in abelian categories comes from the context of 
representable equivalences of module categories \cite{CM,CT}, where the notion of (generalized) $\ast$-module plays a key 
role. Analogous problem in non-abelian case, in particular (generalized) $\ast$-objects and (auto)connected objects, is 
studied in the paper \cite{Mo10}. Note that compact objects play also an important role in triangulated categories \cite[Section 8]{AN}, 
as they are compactly generated, in particular, for the description of Brown representability  \cite{Kr01}. 
Although locally presentable and accessible categories deal with connectedness in the narrow sense \cite{AR, MP},
the motivation and application of the notion is close to the abelian case.

Turning the attention towards the categories of modules, as is shown in \cite{B} and in \cite[1$^{o}$]{Re}, 
small modules can be structurally described in a natural way
by the language of families of submodules:

\begin{lemma}\cite{B,Re}\label{21} The following conditions are equivalent for a module $M$:
\begin{enumerate}
\item[(1)] $M$ is small,
\item[(2)] if $M=\bigcup_{i<\omega}M_n$ for an increasing chain of submodules 
$M_n\subseteq M_{n+1}\subseteq M$, then there exists $n$ such that $M=M_n$,
\item[(3)] if $M=\sum_{i<\omega}M_n$ for a family of submodules 
$M_n\subseteq M$, $n<\omega$, then there exists $k$ such that $M=\sum_{i<k}M_n$.
\end{enumerate}
\end{lemma}

Note that the condition (2) implies immediately that every finitely generated
module is small and (3) shows that there are no
countably infinitely generated small modules.
On the other hand, there are natural constructions of infinitely generated small modules:

\begin{example} \rm (1) A union of a strictly increasing chain
of length $\kappa$, for an arbitrary cardinal
$\kappa$ of uncountable cofinality, consisting of small (in particular finitely generated) submodules is small.

(2) Every $\omega_1$-generated uniserial module is small.
\end{example}

A ring over which the class of all small right modules coincides with the class of all finitely generated ones is called {\it right steady}.
Note that the class of all right steady rings is closed under factorization \cite[Lemma 1.9]{CT},
finite products \cite[Theorem 2.5]{Tr95}, and Morita equivalence \cite[Lemma 1.7]{EGT}. 
However, a ring theoretical characterization of steadiness remains an open
problem with partial results concerning right steadiness of certain  natural classes of rings including right noetherian \cite[$7^0$]{Re}, right perfect \cite[Corollary 1.6]{CT}, right semiartinian of finite socle length \cite[Theorem 1.5]{Tr}
countable commutative \cite[11$^{0}$]{Re}, and
abelian regular rings with countably generated ideals \cite[Corollary 7]{ZT}.

The main task of the first half of the paper is presenting two variants of categories of acts over monoids, namely acts and pointed acts (i.e. acts with a zero element), via a joint general categorial language, in particular,  the notion of a $UD$-category is introduced. Section 3 deals with the crucial issue of decompositions in $UD$-categories. Although categories of pointed acts are not extensive, it is proved that $UD$-categories satisfy  weak version of extensivity (Proposition~\ref{mono-extensive}), which ensures uniqueness of decomposition (Proposition~\ref{Hoefnagel}). Since UD-categories contain enough indecomposable objects, 
their necessary basic properties follow, with Theorem~\ref{decomposition} formulating the existence and uniqueness of indecomposable decomposition. A general composition theory of projective objects in a $UD$-category is built in the next section, where the main result of the section, Theorem~\ref{proj-decomp}, characterizes projective objects as coproducts of indecomposable projective objects. Section 5 lists general properties of connected objects in a $UD$-category, in particular, Theorem~\ref{DcompChar} presents a general criterion for connectedness of objects in a UD-category.
Finally, as an application of this theory the characterization of (auto)connected objects in categories of $S$-acts is also provided (Theorem~\ref{ThmAuto},  Proposition~\ref{PropAuto}).

\section{Axiomatic description of categories of acts}

Before we start the study of common categorial properties of classes of acts over monoids, let us recall some necessary terminology and notation.

Let $\mathcal{C}$ be a category. Denote by $\Mor_\mathcal{C}(A,B)$ the class of all morphisms $A\to B$ in $\mathcal{C}$ for every pair of objects $A,B$ of $\mathcal{C}$; in case $\mathcal{C}$ is clear from the context, the subscript will be omitted.
A monomorphism (epimorphism) in $\mathcal{C}$ is a left (right)-cancellable morphism, i.e., a morphism $\mu$ such that $\mu\alpha = \mu\beta$ ($\alpha\mu = \beta\mu$) implies $\alpha = \beta$. A morphism is a bimorphism, if it is both mono- and epimorphism. A category is balanced, if bimorphisms are isomorphisms (the reversed inclusion holds in general). An object $\theta$ is called initial provided $|\Mor(\theta,A)|=1$ for each object $A$. The category is (co)product  complete if the class of objects is closed under all (co)products. Note that any coproduct complete category contains an initial object, which is isomorphic to $\coprod\emptyset$.
A pair $(\Cc,U)$ is said to be a \emph{concrete category} over the category of sets, which is denoted $\Set$ in the whole paper, if $\Cc$ is a category and $U: \Cc\to\Set$ is a faithful functor.
Finally, a \emph{family} of objects means any discrete diagram and the phrase \emph{the universal property of a coproduct} refers to the existence of unique morphism from a coproduct.

Let $\mathcal S=(S,\cdot, 1)$ be a monoid and $A$ a nonempty set. If there is a mapping $\mu: S \times A \rightarrow A$ satisfying the following two conditions: $\mu\left( 1, a\right) = a$ and $\mu\left( s_2, \mu\left( s_1, a\right) \right) = \mu\left( s_2\cdot s_1, a\right)$ then the pair $(A,\mu)$ of a set $A$ along with the left action $\mu$ is said to be a \emph{left $S$-act}. For simplicity, $\mu\left( s,a\right) $ is often written as $s\cdot a$ or $sa$ and an act  and the act $(A,\mu)$ is denoted $_SA$. 
A mapping $f: {}_{S}A \rightarrow {}_{S}B$ is a homomorphism of $S$-acts, or an $S$-homomorphism provided $f\left( sa\right) = s f\left( a\right)$ holds for any $s\in S, a \in A$.
We denote by
$S$-$\A$ the category of all left $S$-acts with homomorphisms of $S$-acts. Let us point out that we consider the empty $S$-act $_S\emptyset$ to be an (initial) object of $S$-$\A$. Henceforth this category corresponds to the category $S$-$\overline{\A}$ in \cite{KKM}.

Let $\mathcal S$ be a monoid  containing a (necessarily unique) zero element $0$. Then an $S$-act $A$ containing a fixed element $0_A$ which satisfies the axiom $0a=0_A$ for each $a\in A$ is called a \emph{pointed left act} and  then the category of all pointed left $S$-acts with homomorphisms of $S$-acts compatible with zero as morphisms will be denoted by $S$-$\A_0$. Observe that $\{0\}$ is an initial object of the category $S$-$\A_0$.

Recall that both categories $S$-$\oA$ and $S$-$\A_0$ are complete and cocomplete \cite[Remarks II.2.11, Remark II.2.22]{KKM},
in particular, a coproduct of a family of objects $(A_i, i\in I)$ is
\begin{enumerate}
\item[(i)] a disjoint union $\coprod_{i\in I} A_i =\dot{\bigcup} A_i $  in  $S$-$\oA$ by \cite[Proposition II.1.8]{KKM} and

\item[(ii)] $\coprod_{i\in I} A_i =\{(a_i)\in\prod_{i\in I} A_i |\ \exists j: a_i=0 \forall i\ne j\}$ in  $S$-$\A_0$ by \cite[Remark II.1.16]{KKM}.
\end{enumerate}
Furthermore, if we denote  the natural forgetful functor from $S$-$\oA$ or $S$-$\A_0$ into $\Set$ by $U$ (which maps an act to the underlying set of elements and an $S$-homomorphism to the corresponding mapping between sets) both $(S$-$\A_0,U)$ and $(S$-$\oA,U)$ are concrete categories over $\Set$.
Note that  an act can be defined
as a functor from a single-object category to the category of pointed sets and a
pointed act can be defined as a point-preserving functor from a single-object
pointed category to the category of pointed sets where morphisms are natural transformations. In consequence,
 the category of acts is a presheaf category, while
 the category of pointed acts forms a
point-enriched version of a presheaf category, in the sense of enriched category theory.

Let $\mathcal C$ be a coproduct complete category with an initial object $\theta\cong\coprod\emptyset$.
An object $A\in\mathcal{C}$ is called {\it indecomposable} if it is not isomorphic to an initial object nor to a coproduct of two non-initial objects. Note that cyclic acts present natural examples of indecomposable objects in both categories $S$-$\oA$ and $S$-$\A_0$. Nevertheless, the class of indecomposable acts can be much larger,
e. g. the rational numbers form a non-cyclic indecomposable $(\mathbb Z,\cdot)$-act.

As we have declared, the main motivation of the present paper is to describe and investigate connectedness properties of categories of acts over monoids in the general categorial language. 
In particular, we focus on the categories $S$-$\oA$ and $S$-$\A_0$.
The key feature of both of these categories is the existence
of a unique decomposition of every object into indecomposable objects,
which is proved in \cite[Theorem I.5.10]{KKM} for the case of the category $S$-$\oA$. 

First of all, we list several natural categorial properties which ensure an easy handling of the category, the uniqueness of decomposition and provide the existence condition as well. Recall that a pair $(S,\nu)$ is said to be a subobject of an object $A$ if $S$ is an object and $\nu: S\to A$ is a monomorphism.

We say that a concrete category $(\mathcal C,U)$ over the category $\Set$ is a {\it UD-category} (unique decomposition) if the following conditions hold:
\begin{enumerate}
\item[(UD1)] $\mathcal C$ is a coproduct complete balanced category with an initial object $\theta\cong \coprod\emptyset$, for which each morphism $\theta \to A$ is a monomorphism and there is at most one morphism $A \to \theta$,
\item[(UD2)] for any morphism $f\in \Mor(A, B)$ in $\mathcal C$, there exists a subobject $(A^f,\iota)$ of $B$ such that $U(A^f)=U(f)(U(A))\subseteq U(B)$ and $U(\iota)\in \Mor(U(A^f),U(B))$ is the subset inclusion map,
\item[(UD3)]  for each morphism $f\in \Mor(A, B)$ and every subobject $(S,\nu)$ of $B$ such that $U(f)(U(A))\subseteq U(\nu)(U(S))$, there exists a morphism $g\in \Mor(A, S)$ such that $f=\nu g$,
\item[(UD4)] for every family $(A_i,\nu_i)_{i\in I}$ of subobjects of an object $A$, there exist subobjects denoted by $(\bigcap_iA_i^{\nu_i},\iota_\cap)$ and $(\bigcup_iA_i^{\nu_i},\iota_\cup)$ such that
\[
U(\bigcap_iA_i^{\nu_i})=\bigcap_iU(\nu_i)(U(A_i))=\bigcap_iU(A_i^{\nu_i}), \
U(\bigcup_iA_i^{\nu_i})=\bigcup_iU(\nu_i)(U(A_i))=\bigcup_iU(A_i^{\nu_i})
\] 
and both $U(\iota_\cap)$, $U(\iota_\cup)$ are the corresponding subset inclusion mappings,

\item[(UD5)] if $(A,(\nu_0,\nu_1))$ is a coproduct of a pair of objects $(A_0,A_1)$,
then $\nu_0$ and $\nu_1$ are monomorphisms and $\bigcap_{i=1}^2 A_i^{\nu_i}$ is isomorphic to $\theta$,
\item[(UD6)] for every object $A$ and every $x\in U(A)$ there exists an indecomposable subobject $(B, \nu)$ of $A$ such that $x\in U(B^\nu)=U(\nu)(U(B))\subseteq U(A)$.
\end{enumerate}

Any monomorphism $\iota: A\to B$ such that $U(\iota)$ is the subset inclusion map is called \emph{inclusion morphism}.
Note that we will use the notation $(A^f,\iota)$ from (UD2) and $(\bigcup_iA_i^{\nu_i},\iota_\cup)$ from (UD4) freely without other explanations. Moreover, we will write $A_0^{\nu_0}\cap A_1^{\nu_0}$ ($A_0^{\nu_0}\cup A_1^{\nu_0}$ respectively) instead of $\bigcap_{i=1}^2 A_i^{\nu_i}$ ($\bigcup_{i=1}^2 A_i^{\nu_i}$ respectively).

First we make an elementary but frequently used (sometimes without reference) observation:

\begin{lemma}\label{mono-epi}
Let $\psi$ be a morphism of a UD-category $(\Cc,U)$. 
\begin{enumerate}
\item[(1)] If $U(\psi)$ is injective, then $\psi$ is a monomorphism.
\item[(2)] If $U(\psi)$ is surjective, then $\psi$ is an epimorphism.
\item[(3)] $\psi$ is an isomorphism if and only if $U(\psi)$ is a bijection.
\end{enumerate}
\end{lemma}
\begin{proof} Since injective maps are monomorphisms and surjective maps are epimorphisms in $\Set$, (1) and (2) follow immediately from the hypothesis that $U$ is a faithful functor.

(3) The direct implication follows from the well-known fact that any functor preserves isomorphisms and the  reverse one follows from (1) and (2) since $\Cc$ is a balanced category.
\end{proof}

As an easy consequence we obtain a natural property of subobjects in UD-categories:

\begin{lemma}\label{subobject} Let $(B,\nu)$ be a subobject of an object $A$ in a  UD-category $(\Cc,U)$. If $(B^\nu,\iota)$ is a subobject with the inclusion morphism $\iota$ from (UD2) and $\tilde{\nu}\in\Mor(B,B^{\nu})$ from (UD3) satisfying $\nu=\iota\tilde{\nu}$, then $\tilde{\nu}$ is an isomorphism and so $U(\nu)$ is injective.
\end{lemma}
\begin{proof} Since $U(\tilde{\nu})(U(B))=U(\iota)U(\tilde{\nu})(U(B))=U(\nu)(U(B))=U(B^\nu)$ by (UD2), the morphism $\tilde{\nu}$ is an epimorphism by Lemma~\ref{mono-epi}(2). As $\tilde{\nu}$ is a monomorphism and $\Cc$ is a balanced category, $\tilde{\nu}$ is an isomorphism. Since $U(\tilde{\nu})$ is a bijection by 
Lemma~\ref{mono-epi}(3), $U(\tilde{\nu}) = U(\iota) U(\tilde{\nu})$ is injective.
\end{proof}

Let us note that both categories of acts treated in this paper satisfy the previous axiomatics:

\begin{example}\label{act-UD}\rm (1) Let $\mathcal S=(S,\cdot, 1)$ be a monoid. We show that all conditions (UD1)--(UD6) are satisfied by $(\mathcal S$-$\oA,U)$ for the natural forgetful functor $U:S$-$\oA\to \Set$, hence it is a UD-category. 

We have already mentioned that $\mathcal S$-$\oA$ is a coproduct complete category  and that $(\mathcal S$-$\oA,U)$ is a concrete category over $\Set$. Furthermore, the empty act $\emptyset$ with the empty mapping represents an initial object
and the empty map is a monomorphism, since there is no morphism of a nonempty act into $\emptyset$.
Since monomorphisms are exactly injective morphisms, epimorphisms are surjective morphisms and isomorphisms are bijections, $\mathcal S$-$\oA$ is (epi,mono)-structured hence a balanced category (cf. \cite[Section 14]{AHS}), which proves (UD1).
Let us put $A^f=f(A)$ for every morphism $f:A\to B$ and note that intersections and unions of subacts forms subacts as well,
then the conditions (UD2), (UD3), (UD4) and (UD5) follow either immediately from the definition of an act or from well-known basic properties (cf. \cite{KKM}), and (UD6) holds true since cyclic acts are indecomposable.

(2) Let $\mathcal S_0=(S_0,\cdot, 1)$ be a monoid with a zero element $0$. Then $\mathcal S_0$-$\A_0$, similarly as in (1) is also a coproduct complete category and $(\mathcal S_0$-$\A_0,U)$ is a concrete category over $\Set$, where $U$ is the forgetful functor. Clearly, the zero object $\{0\}$ with the zero (mono)morphism forms an initial object of the category $\mathcal S_0$-$\A_0$. Since there is exactly one (zero) morphism from an arbitrary object to the zero object, (UD1) holds true. A similar argumentation as in (1) shows that $(\mathcal S_0$-$\A_0,U)$ satisfies also the conditions (UD2)--(UD6), i.e., it is a UD-category.
\end{example}

\begin{example} \rm
	The concrete category $(\Set, {\rm id}_{\Set})$ is a trivial example of a UD-category: conditions (UD1)--(UD5) are clearly satisfied and for (UD6) note that singletons are indecomposable objects.
\end{example}

\begin{example}\rm 
 Observe that the faithful functor $U$ from the definition of UD-category need not preserve coproducts. While coproducts of the category $\mathcal S$-$\oA$ are precisely disjoint unions, which are coproducts also in $\Set$, and so the forgetful functor $U$ preserves coproducts here, coproducts in $\mathcal S_0$-$\A_0$ glue together zero elements, so they do not coincide with coproducts of the category of all sets, hence the forgetful functor $U$ does not preserve coproducts of $\mathcal S_0$-$\A_0$.
\end{example}

\section{Decompositions}

Our first step in describing UD-categories consists in showing the existence of a unique decomposition into coproduct of indecomposable objects.
Categories satisfying some version of unique decomposition property are widely studied. 
Abelian categories with unique finite coproduct decomposition into indecomposable objects are called Krull-Schmidt categories (see e.g. \cite{Kr02}) and they play distinctive role in structural module theory.

Recall an important class of categories, non-abelian in general, in which all objects posses a coproduct decomposition,
 the \emph{extensive categories}, which can be characterized as  categories $\Bc$ with (finite) coproducts which have pullbacks along colimit structural morphisms and in every commutative diagram 
\[	
	\begin{CD}
	X       @>>>    Z     @<<<  Y \\
	@VVV     @VVV       @VVV\\
	A   @>>> A \sqcup B @<<< B
	\end{CD}
\]
the squares are pullbacks if and only if the top row is a coproduct diagram in $\Bc$ (see \cite[Proposition 2.2]{CLW}, cf.also \cite{EC}).
An extensive category $\Bc$ is said to be \emph{infinitary extensive} if the coproduct diagram above is considered also for infinite coproducts.
Note that extensive categories have a wide range of applicability spanning from being a starting point for construction of distributive categories, which seem to be the correct setting for study of acyclic programs in computer science (cf. \cite{Coc}, \cite{CLW}), to the theory of elementary topoi.

\begin{proposition}\label{S-oA_ext}
	The category $S$-$\oA$ is infinitary extensive.
\end{proposition}
\begin{proof}
Since $S$-$\oA$ is a complete category, there exists any pullback.

Let us for $i\in I$ consider the following diagram in $S$-$\oA$ with pullback squares
\[
\begin{CD}
	A_i       @>\alpha_i>>    A   \\
	@V f_i VV     @VV f V       \\
	B_i   @>\nu_i>> \coprod_{i\in I} B_i 
	\end{CD}
\]
and let us prove $A$ is a coproduct of $A_i$'s. Since we can by \cite[2.2, 2.5]{KKM} consider 
\[
A_i = \{(b, a) \in B_i \times A \, | \, \nu_i(b) = f(a)\}\ \text{ and}\  \coprod B_i = \dot{\bigcup}_{i \in I}\nu_i(B_i),
\]
 then for each $a$ there exist $i \in I, b\in B_i$ with $\nu_i(b) = f(a)$, so $\alpha_i((b,a)) = a$. In consequence $A = \bigcup \alpha_i(A_i)$. Assume $a \in \alpha_{i_0}(x_{i_0}) \cap \alpha_{j_0}(x_{j_0})$; then $f(a) = f \alpha_{i_0}(x_{i_0}) \in \nu_{i_0}(B_{i_0})$, hence $f(a) \in \nu_{i_0}(B_{i_0}) \cap \nu_{j_0}(B_{j_0}) = \emptyset$, so $\alpha_i(A_i) \cap \alpha_j(A_j) = \emptyset$ for $j \ne i$ and $A \simeq \coprod_{i\in I} A_i$.	
\end{proof}

\begin{example}\label{Z-act}
\rm
(1) Consider the monoid  of integers $\mathcal Z = (\mathbb Z,\cdot, 1)$. Then the $\mathcal Z$-act $A=2\mathbb Z\cup 3\mathbb Z$ 
shows that the category $S$-$\A_0$ is not extensive (and so it is not infinitary extensive). See the following commutative diagram with pullback squares where the top row is not a coproduct diagram:
\[
\begin{CD}
A_3       @>\alpha_3>>    2\mathbb{Z} \cup 3\mathbb{Z}    @<\alpha_2<<  A_2 \\
@V \tilde{\pi} VV     @VV\pi V       @VV \overline{\pi} V\\
\mathbb{Z}_3   @>\nu_3>> \mathbb{Z}_3 \sqcup \mathbb{Z}_2 @<\nu_2<< \mathbb{Z}_2,
\end{CD}
\]
where $A_2 = \{(0, a) \, | \, a \in 6\mathbb{Z}\} \cup \{(1, a) \, | \, a \in 3 + 6\mathbb{Z}\} $, $A_3 = \{(0, a) \, | \, a \in 6\mathbb{Z}\} \cup \{(1, a) \, | \, a \in 2 + 6\mathbb{Z}\} \cup \{(2, a) \, | \, a \in 4 + 6\mathbb{Z}\} $, $\alpha_i$ denotes projections on the second coordinate, while $\tilde{\pi}, \overline{\pi}$ on the first one. If we put $\mathbb{Z}_3 \sqcup \mathbb{Z}_2=\{(a,0)\mid a\in \mathbb{Z}_3\}\cup \{(0,b)\mid b\in \mathbb{Z}_w\}$, the morphism $\pi$ can be described by the conditions $\pi^{-1}(a,0)=2a+6\mathbb{Z}$ and  $\pi^{-1}(0,b)=3b+6\mathbb{Z}$.

(2) The category \emph{Top} of topological spaces with continuous maps is extensive, but it is not UD, since it does not satisfy the condition (UD1): the inverse of a continuous bijection need not be continuous, hence \emph{Top} is not balanced. 
\end{example}

Although UD-categories are not extensive in general, we show that all objects of a UD-category satisfy dual version of the mono-coextensivity condition introduced in the paper \cite{H}, which allows us to prove that every object can be uniquely decomposed.

We suppose in the sequel that $(\Cc,U)$ is a UD-category, i.e., a concrete category over $\Set$ satisfying all axioms (UD1)--(UD6) and the notions of objects and morphisms refer to objects and morphisms of the underlying category $\Cc$.
First, we prove a key observation that the description  of coproducts in both categories of acts \cite[Proposition II.1.8, Remark II.1.16]{KKM} can be easily unified within the context of UD-categories.

If $\Ac$ is a family of objects, the corresponding coproduct will be designated $(\coprod\Ac,(\nu_A)_{A \in \Ac})$ where $\nu_A$  is said to be the \emph{structural morphisms} of a coproduct for each $A\in\Ac$.

\begin{proposition}\label{coprod=uni}
Let $\mathcal A=(A_i)_{i\in I}$ be a set of objects and $(\coprod_j A_j, (\nu_i)_{i\in I})$ be a coproduct of $\mathcal A$. Then there exist an inclusion morphism  $\iota_\cup\in\Mor(\bigcup_{i\in I}A_i^{\nu_i}, \coprod_{i\in I} A_i)$ and  morphisms $\mu_j\in\Mor(A_j,\bigcup_{i\in I}A_i^{\nu_i})$  satisfying $\iota_\cup\mu_j = \nu_j$ for each $j$. Furthermore, $(\bigcup_iA_i^{\nu_i}, (\mu_i)_i)$ is a coproduct of $\mathcal A$ and $U(\coprod_{i\in I} A_i)=\bigcup_{i\in I}U(A_i^{\nu_i})$.
\end{proposition}

\begin{proof} Note that the inclusion morphism  $\iota_\cup\in\Mor(\bigcup_iA_i^{\nu_i}, \coprod A_i)$ exists by (UD4) and  morphisms $\mu_j\in\Mor(A_j,\bigcup_iA_i^{\nu_i})$ with $\iota_\cup\mu_j = \nu_j$ exist by (UD3) for each $j$. Using the universal property of a coproduct, we obtain a morphism $\varphi$ such that $\varphi\nu_j = \mu_j$ for every $j$, i.e. the left square of the diagram
\[
\begin{CD}
A_j @>\nu_j>>  \coprod A_i \, @=\coprod A_i\\
@| @VV\varphi V   @AA \iota_\cup A  \\
A_j    @>\mu_{j}>> \bigcup_iA_i^{\nu_i} @= \bigcup_iA_i^{\nu_i}
\end{CD}
\]
commutes in $\mathcal{C}$ and we will show that the right square commutes as well.

Since $\iota_\cup\varphi\nu_j = \iota_\cup\mu_j= \nu_j$ for each $j$, we get again 
by the colimit universal property that $\iota_\cup\varphi = \id_{\coprod A_i}$. Since $\varphi$ has a left inverse which is a monomorphisms, both
$\iota_\cup$ and $\varphi$ are isomorphisms. Hence $U(\coprod A_i)=U(\iota_\cup)(\bigcup_i A_i^{\nu_i})=\bigcup_iU(A_i^{\nu_i})$ and  $(\bigcup_iA_i^{\nu_i}, (\mu_i)_i)=(\bigcup_iA_i^{\nu_i}, (\varphi\nu_i)_i)$ is a coproduct of the family $\Ac$.
\end{proof}

Let $A$ be an object, and $(A_j,\iota_j)$ be subobjects such that $\iota_j$ is the inclusion morphism for each $j\in J$. We say that $((A_j,\iota_j),j\in J)$ is a {\it decomposition} of $A$ if $(A, (\iota_j)_{j\in J})$ is a coproduct of the family $(A_j,j\in J )$. 
Note that we have defined the decomposition for sets of subobjects with inclusion morphisms mainly for clarity of exposition: had we defined it for general subobjects, we would have got the same result using (UD2) and Lemma~\ref{subobject} afterward.

The following easy consequence of Proposition~\ref{coprod=uni} describes a natural decomposition of  a coproduct in $\mathcal{C}$.

\begin{corollary}\label{decomp1} Let $\mathcal A=(A_j,j\in J)$ be a family of objects and $(A, (\nu_j)_j)$ be a coproduct of $\mathcal A$. Then for each $j\in J$ there exists an inclusion morphism $\iota_j\in\Mor(A_j^{\nu_j}, A)$ such that $((A_j^{\nu_j}, \iota_j),j\in J)$ forms a decomposition of $A$.
\end{corollary}

It is well-known that there is a canonical isomorphism    
$\coprod_{i\in I}\left( \coprod \mathcal A_i\right) \cong \coprod\left( \bigcup_{i\in I} \mathcal A_i\right) $
for every family of sets of objects $\mathcal A_i$, $i\in I$
in any  coproduct-complete category with an initial object $\theta$, hence $\theta\sqcup A\cong A$.
Furthermore, let us formulate elementary but useful property of decompositions.

\begin{corollary}\label{decomp2} 
Let $A$ be an object and $\left( \mathcal{A}_i, i \in I\right)$ a family of sets of subobjects $(C,\iota_C)$ of $A$ such that $\iota_C$ is the inclusion morphism and let $B_i=\bigcup_{(C,\iota_C)\in\Ac_i}C^{\iota_C}$ and  $\iota_i$ be the inclusion morphism ensured by (UD4) for each $i\in I$.
The following conditions are equivalent:
\begin{enumerate}
\item For each $i\in I$, $\mathcal{A}_i$ forms a decomposition of the object $B_i$ and  $\left((B_i,\iota_i), i \in I\right) $ is a decomposition of $A$,
\item $\dot{\bigcup}_{i \in I} \mathcal{A}_i$ is a decomposition of $A$.
\end{enumerate}
\end{corollary}

As the morphism of an initial object to an arbitrary object 
is a monomorphism by (UD1), for every object $A$, there exists a subobject denoted by $(\theta_A,\vartheta_A)$ with the inclusion morphism $\vartheta_A$ and $\theta_A\cong\theta$. Note that an initial object $\theta$ has no proper subobject.

It is not a priori clear that the existence of different (even though necessarily isomorphic) initial objects will not cause obstacles. However, the following observations, that  will also turn out to be useful for further dealing with decompositions of objects in a general UD-category, show that the situation is favourable. 

\begin{lemma}\label{ZeroSubobj}
If $A$ is an object and  $(S,\iota)$ is a subobject of $A$ with the inclusion morphism $\iota$, then 
\begin{enumerate}
\item[(1)] $U(\theta_S)=U(\theta_A)$,
\item[(2)] $S\cong\theta$ if and only if $U(S)=U(\theta_A)$ if and only if $U(S)\subseteq U(\theta_A)$.
\end{enumerate}
\end{lemma}

\begin{proof} (1) Since there exists the unique isomorphism $\tau:\theta_S\to \theta_A$ and both $\iota\vartheta_S\tau=\vartheta_A$ and $\iota\vartheta_S=\vartheta_A\tau^{-1}$ are inclusion morphisms, we get the equality $U(\theta_S)=U(\theta_A)$.

(2) If $S\cong\theta$, then  $U(S)=U(\theta_S)=U(\theta_A)$.
The implication  $U(S)=U(\theta_A)$ $\Rightarrow$ $U(S)\subseteq U(\theta_A)$ is clear and, proving indirectly, 
suppose that $S$ is not isomorphic to $\theta$. Then $\vartheta_S$ is a monomorphism by (UD1) which is not an epimorphism. Hence $U(\vartheta_S)$ is not surjective by Lemma~\ref{mono-epi}(2) and so $U(\theta_A)=U(\theta_S)\subsetneq U(S)$ by (1). We have proved that $U(S)\nsubseteq U(\theta_A)$.
\end{proof}
 
Now, we formulate a natural description of decompositions of objects  and of its subobjects.
 
\begin{proposition}\label{decomp3}
Let $A$ be an object and $(A_j,\iota_j)$ be a subobject of $A$ with the inclusion morphism $\iota_j$ into $A$ for every $j\in J$.
Then $((A_j,\iota_j),j\in J)$ is a decomposition of $A$ if and only if
$U(A)=\bigcup_{j}U(A_j)$ and $U(A_i)\cap\bigcup_{j\ne i}U(A_j)=U(\theta_A)$ for each $i\in J$.
\end{proposition}
\begin{proof} Let $((A_j,\iota_j),j\in J)$ be a decomposition of $A$.
Then $U(A)=\bigcup_{j}U(\iota_j)U(A_j)=\bigcup_{j}U(A_j)$ by Proposition~\ref{coprod=uni}. Since 
$((A_i,\iota_i),(\bigcup_{j\ne i}A_j^{\iota_j},\iota))$, where $\iota\in\Mor(\bigcup_{j\ne i}A_j^{\iota_j},A)$ is the inclusion morphism, forms a decomposition of $A$ by Corollary~\ref{decomp2}, (UD3) and (UD4), we get that  $A_i^{\iota_i}\cap(\bigcup_{j\ne i}A_j^{\iota_j})^{\iota}\cong\theta$ by (UD5). Thus $U(A_i)\cap\bigcup_{j\ne i}U(A_j)=U(A_i^{\iota_i})\cap U(\bigcup_{j\ne i}A_j^{\iota_j})^{\iota})=U(\theta_A)$ by Lemma~\ref{ZeroSubobj}.

In order to prove the reverse implication, let us suppose that $U(A)=\bigcup_{j}U(A_j)$ and $U(A_i)\cap\bigcup_{j\ne i}U(A_j)=U(\theta_A)$ for each $i\in J$
and $(\coprod_j A_j, (\nu_j)_{j\in J})$ is a coproduct of the family $(A_j,j\in J)$.
Then there exists a morphism $\varphi\in\Mor(\coprod_j A_j, A)$
such that $\varphi\nu_j=\iota_j$ for all $j$ by the universal property of a coproduct. Since $U(A)=\bigcup_{j}U(A_j)\subseteq U(\varphi)(U(\coprod_j A_j))$, the mapping $U(\varphi)$ is surjective. Let $U(\varphi)(a)=U(\varphi)(b)$ for elements $a,b\in U(\coprod_j A_j)$.
Then there are indexes $j_0,j_1\in J$ and elements $\tilde{a}\in U(A_{j_0})$, $\tilde{b}\in U(A_{j_1})$ for which $a=U(\nu_{j_0})(\tilde{a})$, $b=U(\nu_{j_1})(\tilde{b})$ by Proposition~\ref{coprod=uni}, hence 
\[
\tilde{a}=U(\iota_{j_0})(\tilde{a})=U(\varphi\nu_{j_0})(\tilde{a})=U(\varphi)(a)=U(\varphi)(b) 
= U(\varphi\nu_{j_1})(\tilde{b})=U(\iota_{j_1})(\tilde{b})=\tilde{b}.
\]
Since $\tilde{a}=\tilde{b}\in U(A_{j_0})\cap U(A_{j_1})$,
the hypothesis implies that either $j_0=j_1$ and so 
$a=U(\nu_{j_0})(\tilde{a})=U(\nu_{j_0})(\tilde{b})=b$
or $\tilde{a}=\tilde{b}\in U(\theta_A)$, hence 
$a=U(\nu_{j_0})(\tilde{a})=U(\nu_{j_1})(\tilde{b})=b$
again, as $U(\nu_{j_0})|_{U(\theta_A)}=U(\nu_{j_1})|_{U(\theta_A)}$.
Since $U(\varphi)$ is surjective and injective, $\varphi$ is an isomorphism by Lemma~\ref{mono-epi}(3). Thus $(A, (\varphi\nu_j)_j)$ is a coproduct of the family $(A_j,j\in J)$ and all $\varphi\nu_j$ are the inclusion morphisms, which means that $(A_j,j\in J)$ is a decomposition of $A$.
\end{proof}

Note that the argument of the reverse implication  depends strongly on the fact that $(\mathcal C, U)$ is a concrete category over $\Set$.

\begin{lemma}\label{subobjects-decomp}
Let $A$ be an object, $(B,\mu)$ its subobject with inclusion morphism 
$\mu$ and let $((A_j,\iota_j),j\in J)$ be a decomposition 
of $A$. Then there exists a decomposition $((A_j^{\iota_j}\cap B^{\mu},\mu_j),j\in J)$ of $B$ with inclusion morphisms $\mu_j$, $j\in J$, such that $U(B_j)=U(B)\cap U(A_j)$ for each $j\in J$.
\end{lemma}

\begin{proof} We put $B_j=A_j^{\iota_j}\cap B^{\mu}$
and it is enough to take  
a morphism $\mu_j\in \Mor(B_j,B)$ such that $\mu\mu_j$ is the inclusion 
morphism $B_j\to A$ and $U(B_j)=U(B)\cap U(A_j)$ for each $j\in J$, which 
exists by (UD4) and (UD3). Since $U(\mu\mu_j)=U(\mu)U(\mu_j)$ and $U(\mu)$ 
are inclusions, $\mu_j$ is an inclusion morphism for all $j\in J$. 
Since $\bigcup_{j}U(B_j)=\bigcup_{j}U(B)\cap U(A_j) = U(B)$ and 
\[
U(\theta_A)\subseteq U(B_i)\cap\bigcup_{j\ne i}U(B_j)= U(B)\cap  U(A_i)\cap\bigcup_{j\ne i}U(A_j) =U(\theta_A)
\]
for each $i\in J$, we obtain that $((B_j,\mu_j),j\in J)$ is a decomposition by Proposition~\ref{decomp3}.
\end{proof}

Let us recall the dual version of terminology from \cite{H}.
If $\M$ is a class of morphisms, we say that a commutative square
\[
\begin{CD}
P @>\alpha>>   A \\
@V\beta VV     @VVV \\
B   @>>> M.
\end{CD}
\]
 is \emph{$\M$-pull-back} provided it is a pull-back diagram with $\alpha,\beta\in\M$.
 
We denote by $\M$ the class of all coproduct structural morphisms in the rest of the section. Remark that all coproduct structural morphisms are monomorphisms by (UD5).

An object $M$ is said to be \emph{mono-extensive} if for each pair of morphisms $\gamma,\delta\in\M$ with codomain $M$ there exist
an object $P$ and morphisms $\alpha,\beta\in\M$ with domain $P$
such that 
\[
\begin{CD}
P @>\alpha>>   A \\
@V\beta VV     @VV\gamma V \\
B   @>\delta >> M.
\end{CD}
\]
is $\M$-pull-back and in every commutative diagram 
\[	
	\begin{CD}
	P_1       @>>>    A     @<<< P_2 \\
	@VVV     @VVV       @VVV\\
	B_1   @>>> B_1 \sqcup B_2 @<<< B_2
	\end{CD}
\]
where $M\cong B_1 \sqcup B_2$, the bottom row is a coproduct diagram, and the vertical morphisms belong to $\M$, the top row is coproduct if and only if both the squares are $\M$-pull-backs. Finally, a category is called \emph{mono-extensive} if all objects are mono-extensive.

\begin{lemma}\label{pull-back}
If $\gamma\in\Mor(A,M)$ and $\delta\in\Mor(B,M)$ are monomorphisms, then there exist monomorphisms $\alpha\in\Mor(A^\gamma\cap B^\delta,A)$ and $\beta\in\Mor(A^\gamma\cap B^\delta,B)$ such that $\gamma\alpha=\delta\beta=\iota_\cap$ for the inclusion morphism $i_\cap$ and
\[
\begin{CD}
A^\gamma\cap B^\delta @>\alpha>>   A \\
@V\beta VV     @VV\gamma V \\
B   @>\delta >> M.
\end{CD}
\]
is a pull-back diagram.
\end{lemma}

\begin{proof} Since the existence of morphisms $\alpha\in\Mor(A^\gamma\cap B^\delta,A)$ and $\beta\in\Mor(A^\gamma\cap B^\delta,B)$ follows from (UD3) and (UD4) and both are monomorphisms as $\gamma\alpha=\delta\beta=\iota_\cap$ is a monomorphism, it is enough to check the pull-back universal property. If 
\[
\begin{CD}
\tilde{P} @>\tilde{\alpha}>>   A \\
@V\tilde{\beta} VV     @VV\gamma V \\
B   @>\delta >> M.
\end{CD}
\]
is a commutative diagram, then there exists $\tau\in\Mor(\tilde{P},A^\gamma\cap B^\delta)$ 
satisfying
\[
\gamma\alpha\tau
=\gamma\tilde{\alpha}
=\delta\tilde{\beta}
= \delta\beta\tau
\]
again by (UD3), which implies that 
$\alpha\tau =\tilde{\alpha}$, $\beta\tau =\tilde{\beta}$
 as $\gamma$ and $\delta$ are monomorphisms. Since $\alpha$ and $\beta$ are monomorphisms as well, the same argument gives the uniqueness of $\tau$.
\end{proof}

Let us formulate a description of pull-back diagrams along inclusion morphisms.

\begin{lemma}\label{characterization:pull-back}
Let all morphisms of a commutative diagram
\[
\begin{CD}
P @>>>   A \\
@VVV     @VVV \\
B   @>>> M.
\end{CD}
\]
be inclusion morphisms. Then 
it is a pull-back diagram if and only if $U(P)=U(A)\cap U(B)$.
\end{lemma}

\begin{proof}
The direct implication follows from Lemma~\ref{pull-back} and (UD4), since $P$ and $A^\gamma\cap B^\delta$ are isomorphic.

To prove the reverse implication, let us suppose that the commutative diagram
\[
\begin{CD}
P @>\tilde{\alpha} >>   A \\
@V\tilde{\beta} VV     @VV\gamma V \\
B   @>\delta >> M,
\end{CD}
\]
where all morphisms are inclusion morphisms, satisfies the condition $U(P)=U(A)\cap U(B)$. Since we have a pull-back diagram
\[
\begin{CD}
A^\gamma\cap B^\delta @>\alpha>>   A \\
@V\beta VV     @VV\gamma V \\
B   @>\delta >> M.
\end{CD}
\]
by Lemma~\ref{pull-back}, there exists 
a morphism $\tau\in\Mor(P,A^\gamma\cap B^\delta)$ 
satisfying $\alpha\tau =\tilde{\alpha}$, $\beta\tau =\tilde{\beta}$. Since both mappings $U(\alpha)$ and $U(\tilde{\alpha})$ are inclusions and  $U(\alpha)U(\tau) =U(\tilde{\alpha})$, $U(\tau)$ is an inclusion as well. By the hypothesis $U(\tau)$ is surjective, hence $U(\tau)$ is a bijection  and so $\tau$ is an isomorphism by  Lemma~\ref{mono-epi}(3).
\end{proof}

\begin{lemma}\label{closure}
The class $\M$ of all coproduct structural morphisms is closed under composition, coproducts and pull-back morphisms.
\end{lemma}

\begin{proof} Applying Lemma~\ref{subobject} we may suppose without loss of generality that all considered coproduct structural morphisms are inclusion morphisms.

Let $\phi_0\in \Mor(A_0,B)$ and $\psi\in \Mor(B,C)$ be coproduct structural morphisms. Then there exist $\phi_1\in \Mor(A_1,B)$
and $\phi_2\in \Mor(A_2,C)$ such that
$((A_0,\phi_0), (A_1,\phi_1))$ is a decomposition of $B$ and
$((B,\psi), (A_2,\phi_2))$ forms a decomposition of $C$.
Now it remains to apply Proposition~\ref{decomp3}. 
Since 
\[
U(A_0)\cup U(A_1)\cup U(A_2)=U(B)\cup U(A_2) = U(C)
\]
and 
\[
U(A_i)\cap (U(A_j)\cup U(A_k))=U(\theta)
\]
for all $i\ne j\ne k\ne i$, we can see that 
$
((A_0,\psi\phi_0), (A_1,\psi\phi_1), (A_2,\phi_2))
$
forms a decomposition of $C$, hence $\psi\phi_0\in\M$.

The argument for a coproduct morphism is similar; if
$\phi_i\in \Mor(A_i,B_i)$, $i\in I$, are coproduct structural morphisms, then there exists $\tilde{\phi}_i\in \Mor(\tilde{A}_i,B_i)$  such that
$((A_i,\phi_i), (\tilde{A}_i,\phi_i))$ forms a decomposition of $B_i$ for each $i\in I$. Then $((\coprod_i A_i,\coprod_i\phi_i), (\coprod_i\tilde{A}_i,\coprod_i\phi_i))$ is a decomposition of $\coprod_i B_i$ and since
\[
\bigcup_i U(A_i)\cup \bigcup_i\tilde{A}_i =U(\coprod_i B_i)\ \text{and}\   
\bigcup_i U(A_i)\cap \bigcup_i\tilde{A}_i =U(\theta),
\]
$\coprod_i\phi_i$ is a coproduct structural morphism by Proposition~\ref{decomp3}.

Finally, let 
\[
\begin{CD}
P @>\alpha >>   A \\
@V\beta VV     @VV\gamma V \\
B   @>\delta >> M.
\end{CD}
\]
be a pull-back diagram with $\gamma, \delta\in \M$. Then 
$U(P)=U(A)\cap U(B)$ by Lemma~\ref{characterization:pull-back}
and there exists a decomposition $((A,\gamma),(\tilde{A}, \tilde{\gamma}))$ of $M$. Thus  
$(P,\beta)$ and symmetrically $(P,\alpha)$ are members of decompositions of $M$, hence $\alpha$ ($\beta$, resp.) are coproduct structural morphisms by Lemma~\ref{subobjects-decomp}.
\end{proof}

\begin{proposition}\label{mono-extensive}
Any UD-category is mono-extensive.
\end{proposition}

\begin{proof} The existence of $\M$-pull-backs follows immediately from Lemma~\ref{pull-back} since all coproduct structural morphisms are monomorphisms by (UD5). Let 
\[	
	\begin{CD}
	P_1       @>>>   P_1 \sqcup P_2     @<<< P_2 \\
	@VVV     @VVV       @VVV\\
	B_1   @>>> B_1 \sqcup B_2 @<<< B_2
	\end{CD}
\]
be a commutative diagram with coproduct diagrams on both the rows where all morphisms belong to $\M$. We may suppose without loss of generality by Lemma~\ref{subobject} that
all morphisms are inclusion morphisms and we will show that
 both the squares are $\M$-pull-backs by applying the criterion Lemma~\ref{characterization:pull-back}.
Since, $U(P_i)\subseteq U(B_i)$ for $i=1,2$ and $U(B_1)\cap U(B_2)=U(\theta)$ we obtain
\[
U(P_1)\subseteq U(P_1 \sqcup P_2)\cap U(B_1)
= (U(P_1) \cup U(P_2))\cap U(B_1) =
\]
\[
=(U(P_1)\cap U(B_1)) \cup (U(P_2)\cap U(B_1))=
U(P_1)\cap U(\theta) = U(P_1),
\]
hence $U(P_1)= U(P_1 \sqcup P_2)\cap U(B_1)$ and the left square is a pull-back. The argument for the right square is symmetric.

Conversely, let
\[	
	\begin{CD}
	P_1       @>>>   A     @<<< P_2 \\
	@VVV     @VVV       @VVV\\
	B_1   @>>> B_1 \sqcup B_2 @<<< B_2
	\end{CD}
\]
be a commutative diagram with  all vertical morphisms belonging to $\M$ and with a coproduct diagram in the bottom row such that both the squares are $\M$-pull-backs. By Lemma~\ref{closure} we have that all morphisms of the diagram belong to $\M$. We may suppose that all morphisms are inclusion morphisms again and we need to prove that the top row forms a decomposition of 
the object $A$. Since
$U(P_i)=U(A)\cap U(B_i)$ for $i=1,2$ by Lemma~\ref{characterization:pull-back}, we can easily compute that
\[
U(P_1)\cap U(P_2)= (U(A)\cap U(B_1))\cap (U(A)\cap U(B_2))
= U(A)\cap (U(B_1\cap U(B_2))= U(\theta),
\]
\[
U(P_1)\cup U(P_2)= (U(A)\cap U(B_1))\cup (U(A)\cap U(B_2))
= U(A)\cup (U(B_1\cup U(B_2))= U(A),
\]
by Proposition~\ref{decomp3}, which implies that the top row is a coproduct diagram.
\end{proof}

Recall that an object of a UD-category $(\Cc, U)$ is indecomposable provided it is indecomposable object of $\Cc$ and
let us say that an object $A$ is \emph{uniquely decomposable} if there exist a family of indecomposable objects  $(A_j,j\in J)$
such that $A\cong\coprod_{j\in J}A_j$ and for each family $(\tilde{A}_j,j\in J)$
satisfying $A\cong\coprod_{j\in J}\tilde{A}_j$ there exists a bijection $b:J\to\tilde{J}$ such that $A_j\cong \tilde{A}_{b(j)}$  for each $j\in J$.

\begin{proposition}\label{Hoefnagel}
Every object of a  mono-extensive category possessing a decomposition  into indecomposable objects is uniquely decomposable.
\end{proposition}

\begin{proof} It follows from the dual version of \cite[Remark 2.3]{H} using the dual version of \cite[Theorem 2.1]{H} which can be applied by Lemma~\ref{closure}.
\end{proof}

The next assertion presents a natural construction of indecomposable objects in a UD-category.

\begin{lemma}\label{union of indecomposables}
Let $((A_i,\nu_i), i\in I)$ be a family of subobjects of an object $A$ such that $A_i$ is indecomposable for each $i\in I$.
If $\bigcap_{i\in I}U(A_i^{\nu_i}) \ne U(\theta_A)$, then there exists an inclusion morphism $\iota_\cup\in\Mor(\bigcup_{i\in I} A_i^{\nu_i},A)$  such that $(\bigcup_{i\in I} A_i^{\nu_i},\iota_\cup)$ is an indecomposable subobject of $A$.
\end{lemma}

\begin{proof} 
Put $A^{\prime} = \bigcup_{i\in I} A_i^{\nu_i}$ and let $\iota_\cup\in\Mor(A^{\prime},A)$ be the inclusion morphism ensured by (UD4). Since $(A^{\prime}, \iota_\cup)$ is a subobject, we may suppose without loss of generality that $A=A^{\prime}$.  Remark that the proof repeats the argument of the proof of \cite[Lemma I.5.9]{KKM}. 

Assume that $((B_0, \iota_0), (B_1, \iota_1))$ is a decomposition of $A$ such that $U(B_i)\ne U(\theta_A)$ for both $i=0,1$.
Since $\bigcap_{i\in I}U(A_i^{\nu_i}) \ne U(\theta_A)$ and $U(B_0)\cup U(B_1)=U(A)$ by Proposition~\ref{decomp3}, there exists $j$ for which 
$U(B_j)\cap \bigcap_{i\in I}U(A_i^{\nu_i}) \ne U(\theta_A)$, we may w.l.o.g. assume that $j=0$. Moreover, there exists $i$ such that $U(B_1)\cap U(A_i)\ne U(\theta_{A})$. Thus $U(B_j)\cap U(A_i)\ne  U(\theta_A)$ for both $j=0,1$. 
Then by Lemma~\ref{subobjects-decomp} there exists a decomposition $((\tilde{B}_0,\tilde{\iota}_0),(\tilde{B}_0,\tilde{\iota}_0))$ of $A_i$
such that $U(\tilde{B}_j)\ne U(\theta _A)=U(\theta _{A_i})$ for both $j=0,1$.
Hence  we obtain by Proposition~\ref{decomp3} a contradiction with the hypothesis that $A_i$ is indecomposable.
\end{proof}

Now we can formulate a version of \cite[Theorem I.5.10]{KKM} valid in a general UD-category:

\begin{theorem}\label{decomposition}
Every noninitial object of a UD-category is uniquely decomposable.
\end{theorem}
\begin{proof}
First, we prove the existence of indecomposable decomposition of a noninitial object $A$.

For $a\in U(A)\setminus U(\theta_A)$, which exists by Lemma~\ref{ZeroSubobj}, consider the set 
\[
I_a = \left\lbrace C \, | \, (C, \nu_C) \text{ is an indecomposable subobject of } A \text{ and } a \in U(C)\right\rbrace
\]
and let $(A_a,\iota_a)$ be a subobject with the inclusion map, where $A_a = \bigcup_{(C, \nu_C)\in I_a} C^{\nu_C}$, which exists by (UD4). Then $(A_a,\iota_a)$ is an indecomposable subobject of $A$ by Lemma \ref{union of indecomposables}.

 Furthermore, if $a \neq b$ then either $U(A_a) = U(A_b)$, or $U(A_a) \cap U(A_b)=U(\theta_A)$. Indeed, let $U(A_a) \cap U(A_b)\ne U(\theta_A)$, take $z \in (U(A_a) \cap U(A_b))\setminus U(\theta_A)$, which exists by and (UD4) and consider the indecomposable object $A_z$. Since $z \in U(A_a)$, we have $(A_a,\iota_a) \in I_z$, hence $(A_a,\iota_a)$ is a subobject of $A_z$, similarly for $b$ and vice versa. Therefore $U(A_z) = U(A_a) = U(A_b)$.

Note that for each $a\in U(A)$ there exists an indecomposable subobject $(C, \nu_C)$ of $A$ such that $U(C)$ contains $a$ by (UD6), hence $a\in A_a$.
Moreover, as $A$ is not isomorphic to $\theta$, we get that $U(A)=\bigcup_{a\in  U(A)\setminus U(\theta_A)}U(A_a)$, and we have proved that the representative set of subobjects of the form $(A_x,\iota_x)$ is the desired decomposition. 

The uniqueness follows from Propositions~\ref{mono-extensive} and \ref{Hoefnagel}.
\end{proof}

The following example shows that objects of a mono-extensive category need not be decomposable in general, which illustrates  that the existence part of the assertion of Theorem~\ref{decomposition} depends strongly on the axiom (UD5).

\begin{example}\rm Consider a category whose objects are infinite pointed sets (i.e. sets with one base point $\bullet$) and an initial one-element pointed set $\{\bullet \}$ where morphisms are exactly injective maps compatible with the point $\bullet$.

Then coproducts and pull-backs can be described similarly as in the category of pointed acts: coproducts are disjoint unions with their base points glued together and pull-back diagrams are either of the form 
\[
\begin{CD}
\gamma(A)\cap \delta(B) @>>>   A \\
@VVV     @V\gamma VV \\
B   @>\delta >> M.
\end{CD}\ \ \text{or}\ \ \ \ 
\begin{CD}
\{\bullet \} @>>>   A \\
@VVV     @V\gamma VV \\
B   @>\delta >> M.
\end{CD}
\]
depending on infiniteness or finiteness of the set $\gamma(A)\cap \delta(B)$ (cf. Lemma~\ref{pull-back}).
Thus using similar arguments as in the proof of Proposition~\ref{mono-extensive} we can see that the category is mono-extensive with arbitrary coproducts.

However, since every infinite pointed set can be expressed as a coproduct of two infinite pointed subsets, the category contains no indecomposable object.
\end{example}

\section{Projective objects}

Recall that $(\Cc, U)$ is supposed to be a UD-category. 
We say that an object $P\in \mathcal{C}$ is \textit{projective}, if for any pair of objects $A, B \in \mathcal{C}$ and any pair of morphisms $\pi:A\rightarrow B$, $\alpha: P \rightarrow B$, where $\pi$ is an epimorphism, there exists a morphism $\overline{\alpha}: P \rightarrow A$ in $\mathcal{C}$ such that $\alpha = \pi\overline{\alpha}$, i.e. any diagram 
\[
\begin{CD}
@.   P \, \\
@.   @VV\alpha V \\
A   @>\pi>> B
\end{CD}
\]
in $\mathcal{C}$ with $\pi$ an epimorphism, can be completed into a commutative diagram
\[
\begin{CD}
P @=   P \\
@V\overline{\alpha}VV     @VV\alpha V \\
A   @>\pi>> B.
\end{CD}
\]
Note that the notion of projectivity is one of basic tools of category theory and issue of description of projective objects 
seems to be important task in research of any (concrete) category (see e.g. \cite[Chapter 9]{AHS} or \cite[Section III.17]{KKM}). The main goal of the section is to confirm that the 
structure of projective objects of the underlying category $\Cc$ of a UD-category $(\Cc,U)$ can be described as a coproduct of 
indecomposable projective objects in accordance with the case of categories of acts.

\begin{lemma}\label{coprod-proj}
A coproduct of a family $\left( P_i, i\in I\right)$ of projective objects is projective. 
\end{lemma}
\begin{proof} It is well-known, see e.g. \cite[dual version of 10.40]{AHS}.
\end{proof}

\begin{lemma}\label{proj-summand}
If a coproduct of objects $P = \coprod_I P_i$ is projective, then each object $P_i$, $i\in I$, is projective. 
\end{lemma}
\begin{proof} As $\coprod_j P_i\cong P_i\sqcup (\coprod_{j\ne i}P_j)$, it is enough to prove that for
any pair of objects $P_0,P_1$, if $P_0\sqcup P_1$ is projective, then $P_0$ is projective.

Let the projective situation  
\[
\begin{CD}
@.   P_{0} \, \\
@.     @VV\alpha V \\
A   @>\pi>> B
\end{CD}
\]
be given and let $\lambda_X: X\to P_{0}\sqcup P_1$, for $X \in \left\lbrace P_0, P_1\right\rbrace $, $\mu_X:X\to A\sqcup P_1$ for $X \in \left\lbrace A, P_1\right\rbrace $ and $\nu_X:X\to B\sqcup P_1$ for $X \in \left\lbrace B, P_1\right\rbrace $ be structural coproduct morphisms, which all are monomorphisms by (UD5).
Denote by $\tilde{\alpha}: P_{0}\sqcup P_1 \rightarrow B\sqcup P_1$ a coproduct of morphisms $\alpha: P_0 \rightarrow B$ and $1_{P_1}: P_1 \rightarrow P_1$ which is uniquely determined by the universal property of the coproduct $P_{0}\sqcup P_1$, i.e. $\tilde{\alpha}\lambda_{P_0}=\nu_B\alpha$ and
$\tilde{\alpha}\lambda_{P_1}=\nu_{P_1}$.
Similarly, denote by $ \tilde\pi: A\sqcup P_1 \rightarrow B\sqcup P_1$  a coproduct of morphisms $\pi: A \rightarrow B$ and $1_{P_1}: P_1 \rightarrow P_1$, which means that $\tilde{\pi}\mu_{A}=\nu_B\pi$ and
$\tilde{\pi}\mu_{P_1}=\nu_{P_1}$. 
It is easy to compute applying the universal property of the coproduct $B\sqcup P_1$ that $\tilde\pi$ is an epimorphism since both $\pi$ and $1_{P_1}$ are epimorphisms.

Hence we obtain another projective situation: 
\[
\begin{CD}
@.   P_{0}\sqcup P_1 \, \\
@.     @VV\tilde\alpha V \\
A\sqcup P_1   @>\tilde\pi>> B\sqcup P_1
\end{CD}
\]
By the assumption, there exists a morphism $\varphi\in \Mor(P_{0}\sqcup P_1,A\sqcup P_1)$ such that $\tilde\pi\varphi=\tilde\alpha$.
Let us show that $U(\varphi\lambda_{P_0})(U(P_0))\subseteq U(\mu_A)(U(A))$. 

By (UD2), (UD3) and (UD4) there exists a subobject $(S,\iota)$ of $ A\sqcup P_1$ with the inclusion morphism $\iota$ satisfying $U(S)= U(\varphi\lambda_{P_0})(U(P_0))\cap U(\mu_{P_1})\left( U(P_1)\right)$. Since $\mu_{P_1}$ is a monomorphism, there exists a monomorphism $\sigma: S \rightarrow P_1$, such that $\mu_{P_1}\sigma=\iota$. Then $\tilde{\pi}\iota=\tilde{\pi}\mu_{P_1}\sigma=\nu_{P_1}\sigma$ is a monomorphism. 
Since 
\[
U(\tilde{\pi}\iota)(U(S))\subseteq U(\tilde{\pi})U(\varphi\lambda_{P_0})(U(P_0)) = U(\tilde{\alpha}\lambda_{P_0})(U(P_0))=  U(\nu_B\alpha)(U(P_0)) \subseteq U(\nu_B)(U(B))
\]
and 
\[
U(\tilde{\pi}\iota)(U(S)) = U(\nu_{P_1}\sigma)(U(S)) \subseteq U(\nu_{P_1})(U(P_1)),
\]
by Proposition~\ref{decomp3} and Proposition~\ref{coprod=uni} $U(\tilde{\pi}\iota)(U(S))=U(\theta_{B\sqcup P_1})$ and so $\tilde{\pi}\iota$ factorizes through the morphism $\vartheta: \theta\to B\sqcup P_1$. 
Thus $S\cong \theta$, which implies that 
\[
U(\varphi\lambda_{P_0})(U(P_0))\cap U(\mu_{P_1})( U(P_1))=U(\theta_{B\sqcup P_1}).
\] 
Since $U(\varphi\lambda_{P_0})(U(P_0))\subseteq U(\mu_{A})(U(A))\cup U(\mu_{P_1})( U(P_1))$ by Proposition~\ref{decomp3} and Proposition~\ref{coprod=uni}, we get that $U(\varphi\lambda_{P_0})(U(P_0))\subseteq U(\mu_A)(U(A))$.
In consequence, by (UD3) there exists a morphism $\tau: P_0 \rightarrow A$ such that $\mu_A \tau = \varphi\lambda_{P_0}$; therefore $\tilde{\pi}\varphi\lambda_{P_0} = \tilde{\pi}\mu_A \tau = \nu_B\pi\tau$ and on the other hand $\tilde{\pi}\varphi\lambda_{P_0} = \tilde{\alpha}\lambda_{P_0} = \nu_B\alpha$. Finally, as $\nu_B\pi\tau = \nu_B \alpha$ and the morphism $\nu_B$ is 
a monomorphism by (UD5), we have $\pi\tau = \alpha$.
\end{proof}

Now we are ready to characterize projective objects  of UD-categories:

\begin{theorem}\label{proj-decomp}
An object of a UD-category is projective if and only if it is isomorphic to a coproduct of indecomposable projective objects.
\end{theorem}
\begin{proof} If an object is projective, it possesses a decomposition by Theorem~\ref{decomposition}, which consists of projective objects  by Lemma~\ref{proj-summand}. 
The reverse implication follows immediately from Lemma~\ref{coprod-proj}.
\end{proof}

Let $A$ and $B$ be a pair of objects. 
Recall that $B$ is a {\it retract of} $A$ if there are morphisms $f\in\Mor(A,B)$ and  $g\in\Mor(B,A)$ such that $fg=1_B$.
The morphism $f$ is then called {\it retraction} and $g$ {\it coretraction}. Note that each retraction is an epimorphism and each coretraction is a monomorphism. An object $G$ of a category is said to be \textit{generator} if for any object $A \in \mathcal{C}$ there exists an index set $I$ and an epimorphism $\pi: \coprod_{i\in I} G_i \rightarrow A$ where $G_i \simeq G$.

\begin{lemma}\label{generator-retract}
If $\mathcal C$ contains a generator $G$, then every indecomposable projective object is a retract of $G$.
\end{lemma}
\begin{proof} We generalize arguments of \cite[Propositions III.17.4 and III.17.7]{KKM}.

Let $P$ be an indecomposable projective object. Since $G$ is a generator, there are a coproduct $\coprod_iG_i$ of objects $G_i\cong G$ with structural morphisms $\nu_i\in\Mor(G_i,\coprod_iG_i)$ and an epimorphism $\pi\in \Mor (\coprod_iG_i,P)$.
Moreover, there exists a (mono)morphism $\gamma\in\Mor(P,\coprod_iG_i)$ for which $\pi\gamma=1_P$ due to the projectivity of $P$.
Note that $P\cong P^\gamma$ by Lemma~\ref{mono-epi} and 
there exists a decomposition $((H_i,\mu_i), i\in I)$ of $P^\gamma$ for which 
$U(H_i)=U(P^\gamma)\cap U(G_i^{\nu_i})=U(\gamma)(U(P))\cap U(\nu_i)(U(G_i))$ for each $i\in I$ by Lemma~\ref{subobjects-decomp}.
As $P^\gamma$ is indecomposable, there exists an $i\in I$  such that $U(\gamma)(U(P))=U(P^\gamma)\subseteq U(\nu_i)(U(G_i))$ by Proposition~\ref{decomp3}, hence there exists a morphism $\varphi\in \Mor(P,G_i)$ such that $\nu_i\varphi=\gamma$ by (UD3). 
Thus $\pi\nu_i\varphi=\pi\gamma=1_P$ which shows that $\pi\nu_i$ is the desired retraction.
\end{proof}

\section{Connected objects}

In this section we describe connected objects in
a UD-category $(\mathcal C,U)$.

Let $C$ be an object, $\mathcal A=(A_i,i\in I)$ a family of objects and $(\coprod_{i\in I} A_i, \{\nu_i\}_{i\in I})$ a coproduct of the family $\Ac$. Using the covariant functor $\Mor(C,-)$ from $\Cc$ to $\Set$, we define a natural morphism in the category $\Set$ 
\[
\Psi_\mathcal A^C: \coprod_{i\in I}\Mor(C,A_i) \to\Mor(C,\coprod_{i\in I} A_i) 
\]
which is the unique morphism such that the following square is commutative for all $i\in I$
\[
\begin{CD}
\Mor\left(C, A_i\right) @>\mu_i>> \coprod_I \Mor\left(C, A_i\right) \\
@V \Mor\left( C, \nu_i\right) VV  @V\Psi^C_\mathcal A VV \\
\Mor\left(C, \coprod_I A_i\right) @=\Mor\left(C, \coprod_I A_i\right)\\
\end{CD}
\]
where $\mu_i: \Mor\left(C, A_i\right) \rightarrow \coprod_I \Mor\left(C, A_i\right)$ is a coproduct structural inclusion in $\Set$.
Since coproducts of objects in $\Set$ are isomorphic to disjoint unions of the corresponding objects, we have $\coprod_I \Mor\left(C, A_i\right) = \dot{\bigcup}\, \Mor\left(C, A_i\right)$ and we can describe $\Psi_\mathcal A^C$ explicitly as $\Psi_\mathcal A^C(\alpha)=\nu_i\alpha$ for each index $i$ satisfying $\alpha\in \Mor(C, A_i)$.

It is worth mentioning that it is natural to consider morphisms 
\[
\tilde{\Psi}_\mathcal A^C: \coprod_{i\in I}\Mor(U(C),U(A_i)) \to\Mor(U(C),U(\coprod_{i\in I} A_i))= \Mor(U(C),\bigcup_{i\in I}U( A_i))
\]
as we deal with concrete category. Since $U$ is a faithful functor, such a concept is equivalent to original one, however it seems to be technically more difficult.

\begin{lemma}\label{Psi} Let $C$ be an object, $I$ an index set consisting of at least two elements,  $\mathcal A=\{A_i\mid i\in I\}$ a family of objects. Suppose that $(A, \{\nu_i\}_{i\in I})$ is a coproduct of the family $\Ac$ and $\alpha, \beta\in \coprod_{i\in I} \Mor\left(C, A_i\right)= \dot{\bigcup}_{i\in I}\, \Mor\left(C, A_i\right)$.
\begin{enumerate}
\item[(1)] If $\alpha\ne \beta$ and $\Psi_\mathcal A^C(\alpha)=\Psi_\mathcal A^C(\beta)$, 
then there exist $i\ne j$ such that 
$\alpha\in \Mor\left(C, A_i\right)$, $\beta\in \Mor\left(C, A_j\right)$ and
 $U(\nu_i\alpha)(U(C))=U(\theta_{A})=U(\nu_j\beta)(U(C))$.
\item[(2)] If $i,j\in I$ and  $\alpha\in \Mor\left(C, A_i\right)$ and $\beta\in \Mor\left(C, A_j\right)$ such that 
$U(\alpha)(U(C))=U(\theta_{A_i})$ and 
$U(\beta)(U(C))=U(\theta_{A_j})$, then $\Psi_\mathcal A^C(\alpha)=\Psi_\mathcal A^C(\beta)$.
\item[(3)] $\Psi_\mathcal A^C$ is injective (i.e. it is a monomorphism in the category $\Set$) if and only if $\Mor\left(C, \theta\right)=\emptyset$.
\end{enumerate}
\end{lemma}
\begin{proof} 
(1) If there exists an $i$ for which $\alpha, \beta\in \Mor\left(C, A_i\right)$, then $\nu_i\alpha=\nu_i\beta$, hence $\alpha=\beta$ as $\nu_i$ is a monomorphism by (UD5).
In consequence, the hypotheses $\alpha\ne \beta$ and 
$\Psi_\mathcal A^C(\alpha)=\Psi_\mathcal A^C(\beta)$ 
imply that there exists  $i\ne j$ such that 
$\alpha\in \Mor\left(C, A_i\right)$, $\beta\in \Mor\left(C, A_j\right)$, so we get
\[
U(\nu_i\alpha)(U(C))=U(\Psi_\mathcal A^C(\alpha))(U(C))=U(\Psi_\mathcal A^C(\beta))(U(C))=U(\nu_j\beta)(U(C)).
\]
Since $U(\nu_i)(U(A_i))\cap U(\nu_j)(U(A_j))=U(\theta_{A})$ by Corollary~\ref{decomp1} and Proposition~\ref{decomp3} and since 
\[
U(\nu_i\alpha)(U(C))=U(\nu_j\beta)(U(C))\subseteq U(\nu_i)(U(A_i))\cap U(\nu_j)(U(A_j))=U(\theta_{A}),
\]
we get that $U(\nu_i)\alpha(U(C))=U(\nu_j\beta)(U(C))=U(\theta_{A}$) by Lemma~\ref{ZeroSubobj}.

(2) Since $U(\nu_i)(U(\theta_{A_i}))=U(\theta_{A})=U(\nu_j)(U(\theta_{A_j}))$, we have
\[
U(\Psi_\mathcal A^C(\alpha))(U(C))=U(\nu_i\alpha)(U(C))=U(\nu_j\beta)(U(C))=U(\Psi_\mathcal A^C(\beta))(U(C))=
U(\theta_{A})
\]
again by the same argument as in (1) using Corollary~\ref{decomp1},  Lemma~{ZeroSubobj} and Proposition~\ref{decomp3}.
As $U$ is a faithful functor, both morphisms $\Psi_\mathcal A^C(\alpha), \Psi_\mathcal A^C(\beta)$ can be viewed as elements of $\Mor(C,\theta_A)$. Since $|\Mor(C,\theta_A)|=|\Mor(C,\theta)|\le 1$ by (UD1), we get the required equality $\Psi_\mathcal A^C(\alpha)=\Psi_\mathcal A^C(\beta)$.

(3) If $\Mor\left(C, \theta\right)\ne\emptyset$, there exists  $\alpha_i\in \Mor(C,A_i)$ such that $U(\alpha_i)(U(C))=U(\theta_{A_i})$ for all $i\in I$ by Lemma~\ref{ZeroSubobj} again. 
Thus $\Psi_\mathcal A^C(\alpha_i)=\Psi_\mathcal A^C(\alpha_j)$ for all $i,j\in I$ by (2), which implies that $\Psi_\mathcal A^C$ is not injective.

On the other hand, if $\Psi_\mathcal A^C$ is not injective, then there exists $i$ and $\alpha\in \Mor\left(C, A_i\right)$ such that 
$U(\nu_i\alpha)(U(C))=U(\theta_{A})$ by (1). As $\theta_{A}\cong\theta$, there exists a morphism in $\Mor\left(C, \theta\right)$ by (UD3).
\end{proof}

Since there is no morphism of a nonempty act $C$ into the empty act 
$\emptyset$,
all mappings $\Psi_\mathcal A^C$ are injective in the category $\mathcal S$-$\oA$ by Lemma~\ref{Psi}(3),
similarly to the case of abelian categories (cf \cite[Lemma~1.3]{KZ}).
Applying the same assertion, we can see that it is not the case of the category  $\mathcal S_0$-$\A_0$.

\begin{example} {\rm
If $\mathcal S_0=(S_0,\cdot, 1)$ is a monoid with a zero element 
(for example $(\mathbb Z,\cdot,1)$), $C$ is a right $S_0$-act and $\mathcal A=\{A_i\mid i\in I\}$ 
is a family of $S_0$-acts contained in the category $\mathcal S_0$-$\A_0$ satisfying
$|\mathcal A|\ge 2$, then the mapping $\Psi_\mathcal A^C$ is not injective
by Lemma~\ref{Psi}(3).

In particular, if we put $C=A_i=\{0\}$ for every $i\in I$, then 
$|\coprod_I \Mor\left(C, A_i\right)|=|I|$ and
$|\Mor\left(C, \coprod_I A_i\right)|=1$, so the mapping $\Psi_\mathcal A^C$ glues together all morphisms of the arbitrarily large set $\coprod_I \Mor\left(C, A_i\right)$.
}
\end{example}

Using the notation of the mapping $\Psi_\mathcal A^C$ we are ready to generalize abelian-category definition of a connected object to UD-categories.

We say that an object $C$ is \emph{$\mathcal D$-connected} (or connected with respect to $\mathcal D$), if the morphism  
$\Psi_\mathcal A^C$
is surjective for each family $\mathcal A$ of objects from the class 
$\mathcal D$ and $C$ is \emph{connected} if it is $O_{\mathcal C}$-connected 
for the class $O_{\mathcal C}$ of all objects of the category $\mathcal C$. Finally, an object $C$ is called \emph{autoconnected}, if it is $ \{C\}$-connected.
Observe that every connected object is $\mathcal D$-connected for an arbitrary class $\mathcal D$ of objects, in particular, it is autoconnected.

Let $\mathcal D$ be a class of objects of the category $\mathcal C$
and denote by $\mathcal D^{\coprod}=\{\coprod_{i} D_i\mid D_i\in \mathcal D\}$ the class of all coproducts of all families of objects of $\mathcal D$.

Let us formulate a non-abelian version of \cite[Proposition 2.1]{GNM} (cf. also \cite[Theorem 2.5]{KZ}):

\begin{theorem}\label{DcompChar} The following conditions are equivalent for an object
$C$ and a class of objects $\mathcal D$:
\begin{enumerate}
\item[(1)] $C$ is $\mathcal D$-connected,
\item[(2)] for each pair of objects $A_1\in \mathcal D$ and 
$A_2\in \mathcal D^{\coprod}$ and each  morphism
$f\in\Mor(C, A_1\sqcup A_2)$ there exists $i\in\{1,2\}$ such that
$f$ factorizes through $\nu_i$,
\item[(3)] for each  pair of objects $A_1\in \mathcal D$ and 
$A_2\in \mathcal D^{\coprod}$ and each morphism
$f\in\Mor(C, A_1\sqcup A_2)$ there exists $i\in\{1,2\}$ such that $U(f)(U(C))\subseteq U(\nu_i)(U(A_i))$,
\item[(4)] for each  family $(A_i, i\in I)$ of objects of $\mathcal D$ and each morphism  
$f\in\Mor(C, \coprod_{i\in I} A_i)$ there exists $i\in I$ such that
$f$ factorizes through $\nu_i$,
\item[(5)] for each  family $(A_i, i\in I)$ of objects of $\mathcal D$ and each morphism  $f\in\Mor(C, \coprod_{i\in I} A_i)$ there exists $i\in I$ such that
$U(f)(U(C))\subseteq U(\nu_i)(U(A_i))$,
\end{enumerate}
where $\nu_i$ denotes the structural morphism of a corresponding coproduct $A_1\sqcup A_2$ or $\coprod_{i\in I} A_i$.
\end{theorem}

\begin{proof} (1)$\Rightarrow$(4) Let $\mathcal A=\{A_i\mid i\in I\}$ be a family of objects of the class $\mathcal D$ and 
$f\in \Mor\left(C,\coprod_{i\in I} A_i\right)$. Since $C$ is $\mathcal D$-connected, the mapping  $\Psi_\mathcal A^C$ is surjective by definition, hence there exists  $i$ and $\alpha\in \Mor\left(C, A_i\right)$ such that $f=\nu_i\alpha$.

(4)$\Rightarrow$(1) Let $f\in \Mor\left(C,\coprod_{i\in I} A_i\right)$ for a family $\Ac = (A_i, i\in I) \subseteq \mathcal{D}$. Then there exists $i\in I$ and $\tilde{f}\in \Mor\left(C, A_i\right)$ such that $f=\nu_i\tilde{f}$, hence $\Psi_\mathcal A^C(\tilde{f})=f$.

(4)$\Rightarrow$(5) Since there exists $i\in I$ and $\tilde{f}\in \Mor\left(C, A_i\right)$ for which $f=\nu_i\tilde{f}$ we get 
\[
U(f)(U(C))=U(\nu_i)U(\tilde{f})(U(C))\subseteq U(\nu_i)(U(A_i)).
\]

(5)$\Rightarrow$(4) It is a direct consequence of (UD3).

The equivalence (2)$\Leftrightarrow$(3) is a special case of   (4)$\Leftrightarrow$(5).

(3)$\Rightarrow$(5) Let $f\in\Mor(C, \coprod_{i\in I} A_i)$
for a  family $(A_i, i\in I)$ of objects of $\mathcal D$,
put $A=\coprod_{i\in I} A_i$ and assume to a contrary that $U(f)U(C)\nsubseteq U(\nu_i)U(A_i)$ for all $i\in I$. Then by (3) $U(f)U(C)\subseteq \bigcup_{i\ne j}U(A_i^{\nu_i})$ for every $j\in I$, hence by Proposition~\ref{decomp3}
\[
U(f)U(C)\subseteq \bigcap_{j\in I}\bigcup_{i\ne j}U(A_i^{\nu_i})=U(\theta_A)\cup 
\bigcap_{j\in I}U(A)\setminus U(A_j^{\nu_j})=U(\theta_A),
\]
a contradiction. 

The implication (4)$\Rightarrow$(2) is clear, since $A\in \mathcal D^{\coprod}$ if and only if there exists a family  $\mathcal A=\{A_i\mid i\in I\}$ of objects of $\mathcal D$ satisfying $A=\coprod_{i\in I} A_i$.
\end{proof}

Let us reformulate the Theorem~\ref{DcompChar} for the particular (but important) case of connectedness:

\begin{corollary}\label{compChar} The following conditions are equivalent for an object $C$:
\begin{enumerate}
\item[(1)] $C$ is connected,
\item[(2)] for every pair of objects $A_1$ and $A_2$ and each  morphism
$f\in\Mor(C, A_1\sqcup A_2)$ there exists $i\in\{1,2\}$ such that
$f$ factorizes through the structural coproduct morphism $\nu_i$,
\item[(3)] for every pair of objects $A_1$ and $A_2$ and each morphism
$f\in\Mor(C, A_1\sqcup A_2)$ either $U(f)(U(C))\subseteq U(\nu_1)(U(A_1))$ or $U(f)(U(C))\subseteq U(\nu_2)(U(A_2))$, where 
$\nu_i$, $i=1,2$, is the structural coproduct morphism.
\end{enumerate}
\end{corollary}

The following description of autoconnectedness presents another consequence of Theorem~\ref{DcompChar}:  

\begin{corollary}\label{autoCompChar} The following conditions are equivalent for an object $C$:
\begin{enumerate}
\item[(1)] $C$ is autoconnected,
\item[(2)] for each morphism $f\in\Mor\left(C,\coprod_{i\in I} C_i\right)$, 
where $C_i\cong C$ for all $i\in I$, there exists an index $i$ 
such that $U(f)(U(C))\subseteq U(\nu_i)(U(C_i))$,
\item[(3)] for each morphism $f\in\Mor\left(C,\coprod_{i\in I} C_i\right)$, 
where $C_i\cong C$ for all $i\in I$, there exists an index $i$ 
such that $U(f)(U(C))\cap U(\nu_j)(U(C_j))=U(\theta_{\coprod  C_i})$ for each $j\ne i$,
\end{enumerate}
where $\nu_i$ are the  structural morphism of a coproduct $\coprod_{i\in I} C_i$.
\end{corollary}

In order to obtain useful characterization of connected objects in a general UD-category we say that an object $B$ is an {\it image} of an object $A$ if there is a morphism 
$\pi\in\Mor(A, B)$ with $U(\pi)$ surjective.
Observe that connected objects in the category $\mathcal C$ are precisely objects whose every image is indecomposable:

\begin{proposition} \label{comp=>indecomp}
A noninitial object $C$ of the category $\mathcal C$ is connected if and only if
every image of $C$ is indecomposable. 
\end{proposition}

\begin{proof}
$(\Rightarrow)$ Let $\pi:C\to \overline{C}$ be a morphism such that  $U(\pi)$ is surjective and $((A_1,\iota_1), (A_2, ,\iota_2))$ is a decomposition  of $\overline{C}$ with $A_1\not\cong\theta\not\cong A_2$. Then $U(\pi)(U(C))=U(\overline{C})\nsubseteq U(\iota_i)(U(A_i))$ for both $i=1,2$ by Proposition~\ref{decomp3}, hence 
$C$ is not connected by Corollary~\ref{compChar}.

$(\Leftarrow)$ 
If $C$ is not connected, then there exists a pair of objects 
$A_1$ and $A_2$ and a morphism $\pi\in\Mor\left(C, A_1\sqcup A_2\right)$ such that 
$U(C^\pi)=U(\pi)(U(C))\nsubseteq U(\nu_i)(U(A_i))$ for $i\in\{1,2\}$ by Corollary~\ref{compChar}, where $\nu_1$, $\nu_2$ are the structural coproduct morphisms. 
Then there exists a morphism $\tilde{\pi}\in\Mor(C,C^\pi)$ with $U(\tilde{\pi})$ surjective and $\iota$ an inclusion morphism 
satisfying $\pi=\iota\tilde{\pi}$ by (UD2) and (UD3). Furthermore, $(U(\nu_1)(U(A_1),U(\nu_2)(U(A_2))$ induces a nontrivial 
decomposition of $C^\pi$ by Lemma~\ref{subobjects-decomp}. We have proved that the image $C^\pi$ of $C$ is not indecomposable.
\end{proof}

Let us mention, as an easy consequence of the last claim, that every connected object in $\mathcal C$ is indecomposable.

The connectedness (smallness) property originally studied in the branch of (left $R$-)modules has been defined in a similar fashion and the notion of self-smallness as a generalization of the property of being finitely generated can be transferred via the notion of an autoconnected object to $UD$-categories and specially to those of $S$-acts. (see e.g. \cite{AM}, \cite{Dv}).

Note that a stronger version of the assertion of Proposition~ \ref{comp=>indecomp} holds true for extensive categories: an object of an extensive category is connected if and only if it is indecomposable by \cite[Theorem 3.3]{EC}.

Using a similar argument as in the direct implication of Proposition~\ref{comp=>indecomp} we get a necessary condition of autoconnected objects:

\begin{lemma}\label{autoIndec}
Any noninitial autoconnected object is indecomposable.
\end{lemma}
\begin{proof} Assume that the autoconnected object $C$ decomposes into $\mathcal{B}=((B_1,\iota_1), (B_2,\iota_2))$.
Let $\nu\in \Mor\left(C,C\sqcup C\right)$ be the morphism satisfying $\nu\iota_i=\nu_i\iota_i$ for $i=1,2$ which exists by the universal property of a coproduct $B_1\sqcup B_2$, where $\nu_1,\nu_2$ are the  structural morphisms of the coproduct $C\sqcup C$. Then there exists $i$ such that $U(\nu)U(C)\subseteq U(\nu_i)U(C)$ by Corollary~\ref{autoCompChar}, w.l.o.g. we may suppose $i=1$.
Then
\[
U(\nu_2)U(B_2^{\iota_2})=U(\nu_2\iota_2)U(B_2)\subseteq 
U(\nu_1)U(C)\cap U(\nu_2)U(C) = U(\theta_{C\sqcup C}),
\]
which implies that $B_2^{\iota_2}\cong \theta$ by Lemma~\ref{ZeroSubobj}.
Thus the decomposition $\mathcal{B}$ is trivial, so $C$ is indecomposable.
\end{proof}

\begin{proposition}\label{endoImage}
For an autoconnected object $C \in \mathcal{C}$ and an endomorphism
 $f\in \Mor\left(C,C\right)$, the object $C^f$ is autoconnected, too.
\end{proposition}
\begin{proof}

Suppose that $C^f$ is not autoconnected. Then by Corollary~\ref{autoCompChar}
there is a morphism $g\in\Mor(C^f, \coprod_{i\in I} D_i)$ such that
$D_i \cong C^f$ and for each $i\in I$ there exists $j\ne i$ such that 
\[
U(g)(U(C^f)) \cap U(\nu_{j_i})(U(D_{j_i})) \neq U(\theta_{\coprod_{i} D_i}),
\]
where $\nu_i$, $i\in I$ denotes a coproduct structural morphism.
Let $(\coprod_{i\in I} C_i, (\mu_i)_{i\in I})$ be the coproduct of objects $C_i\cong C$. 
Then by (UD2) and (UD3) there exist morphisms $\tilde{f}\in \Mor\left(C,C^f\right)$ and $\iota\in \Mor\left(C^f,C\right)$ such 
that $\iota$ is an inclusion morphism and $f=\iota\tilde{f}$, hence we may suppose that there exist inclusion morphisms $\iota_i\in\Mor(D_i,C_i)$ for all $i\in I$.
Then by the universal property of a coproduct there exists $\nu\in\Mor(\coprod_{i\in I} D_i,\coprod_{i\in I} C_i )$ such that the diagram
\[	
	\begin{CD}
	\coprod_{i} D_i     @>\nu >>  \coprod_{i} C_i \\
	@A\nu_j AA     @A\mu_j AA       \\
	D_j   @>\iota_j >> C_j 
	\end{CD}
\]
commutes for all $j\in I$. Note that $U(\nu)$ is injective because 
\textbf{!!!!!!!!} $U(\nu)U(\nu_i)(D_i)\subseteq U(\nu)U(\nu_i)(D_i)$ \textbf{!!!!!!!!},
$U(\nu)U(\theta_{\coprod_{i} D_i})=U(\theta_{\coprod_{i} C_i})$
and $U(\nu)U(\nu_i)=U(\mu_i)U(\iota_i)$ is injective for each $j$.
Thus 
\[
U(\theta_{\coprod_{i} C_i})=U(\nu)U(\theta_{\coprod_{i} D_i})\ne  U(\nu)(U(g)(U(C^f)) \cap U(\nu_{j_i})(U(D_{j_i})))
\subseteq 
\]
\[
\subseteq U(\nu g\tilde{f})(U(C)) \cap U(\mu_{j_i}\iota_{j_i})(U(D_{j_i})))
\subseteq U(\nu g\tilde{f})(U(C)) \cap U(\mu_{j_i})(U(C_{j_i})))
\]
for each $i\in I$, which implies that $C$ is not autoconnected by Corollary~\ref{autoCompChar}.
\end{proof}


\section{Categories of $S$-acts}

Let $\mathcal S=(S,\cdot, 1)$ be a monoid (possibly with zero $0$) through the whole section.
Recall that for  $\mathcal S $ both categories $S$-$\A_0$ and $S$-$\oA$ of $S$-acts are UD-categories by Example \ref{act-UD}. 
We will use basic properties of these categories 
summarized in the axiomatics (UD1)--(UD6) freely in the sequel.
For standard terminology concerning the theory of acts we refer to the monograph \cite{KKM}.

\subsection{Connected acts}

The following consequence of Corollary~\ref{compChar} shows that the reverse implication of \cite[Lemma I.5.36]{KKM} holds true.

\begin{lemma}\label{char.oA} Connected objects in the category $S$-$\oA$ are precisely indecomposable objects.
\end{lemma}
\begin{proof} The assertion follows from \cite[Lemma I.5.36]{KKM} and Corollary~\ref{compChar}.
\end{proof}

Since the category $S$-$\oA$ is infinitary extensive by Proposition~\ref{S-oA_ext}, the previous lemma also follows from \cite[Theorem 3.3]{EC}. 

Recall that a left $S$-act $A$ is called \emph{cyclic} if there exists $a\in A$ for which $Sa=\{sa\mid s\in S\}=A$, and $A$ is called \emph{locally cyclic}
 if for any pair $a,b\in A$ there exists $c\in A$ such that $a,b \in Sc$. Since cyclic acts are locally cyclic and locally cyclic acts are indecomposable, we obtain an immediate
consequence of Lemma~\ref{char.oA}:

\begin{corollary}\label{loc.cyclic} Every locally cyclic left act is connected in the category $S$-$\oA$.
\end{corollary}

Furthermore, we prove a sufficient condition of connectedness for both considered categories of acts.

\begin{proposition}\label{cyclic} Every cyclic left act is connected in both categories $S$-$\oA$ and $S$-$\A_0$.
\end{proposition}

\begin{proof}
By Corollary~\ref{loc.cyclic} we only need to prove the  claim for the category $S$-$\A_0$.
Since any factor of a cyclic act is cyclic, and so indecomposable, the Proposition~\ref{comp=>indecomp} gives the result in the category $S$-$\A_0$.
\end{proof}

The corresponding variant of Lemma~\ref{char.oA} as a criterion of connectedness in the category $S$-$\A_0$ shall deal with all factors of an act, namely, connected objects in the category $S$-$\A_0$ are precisely objects whose every  image is indecomposable by Proposition~\ref{comp=>indecomp}.

The following example shows that in the case of the category $S$-$\A_0$ the implication in Proposition~\ref{cyclic} cannot be inverted  in general:

\begin{example}\label{exmplZ} \rm Let $\mathcal Z = (\mathbb Z,\cdot, 1)$ be a monoid with zero.
 
(1) Consider again $\mathcal Z$-act $A=2\mathbb Z\cup 3\mathbb Z$ from Example~\ref{Z-act}. 
Then $A$ is an indecomposable act which is not connected in the category $S$-$\A_0$. Indeed, if we consider the morphism
$f_6:A\to\mathbb Z_6$ given by $f_6(a)= a\mod6$, then 
the image $f_6(A)=\{0,2,4\}\cup \{0,3\}$ decomposes, hence it is not connected by Proposition~\ref{comp=>indecomp}.

(2) Every abelian group is connected in the category $\mathcal Z-\oA$ since every $\mathcal Z$-subact contains $0$. More generally, for a monoid $S$ with zero, any $A \in S$-$\A_0$ can be recognized as an object of $S$-$\oA$ and it becomes indecomposable in $\mathcal Z-\oA$ , hence connected by Lemma~\ref{char.oA}.
\end{example}

In compliance with \cite[Definition 4.20]{KKM} recall that for a subact $B$ of an act $A$ the \emph{Rees congruence} $\rho_B$ on $A$ is defined by setting $a_1 \rho a_2$ if $a_1 = a_2$ or $a_1, a_2 \in B$. The corresponding factor act $A/B$ is called \emph{Rees factor of A by B} then.

\begin{lemma}\label{noncompRees}
Let $A \in S$-$\A_0$ and $A_1$ and $A_2$ be its proper subacts. If $A = A_1 \cup A_2$ and $A_i \setminus \left( A_1 \cap A_2\right) \neq \emptyset$ for both $i=1,2$, then $A$ is not connected in $S$-$\A_0$.
\end{lemma}
\begin{proof}
Consider the projection $\pi$ of $A$ onto the Rees factor $A/\left( A_1 \cap A_2\right)$, which is decomposable into $\pi(A_1) \sqcup \pi(A_2)$. Now use Corollary~\ref{compChar}.
\end{proof}

Note that a subact $B$ of an act $A$ can be viewed as an subobject $(B,\iota)$ of $A$ with the inclusion morphism $\iota$
and recall that a subact $B$ of a left $S$-act $A$ (in  $S$-$\oA$ or $S$-$\A_0$) is called \emph{superfluous} if $B \cup C \neq A$ for any proper subact $C$ of $A$ (see \cite[Definition 2.1]{Kho-Ro}). An act is called \emph{hollow} if each of its proper subacts is superfluous (see \cite[Definition 3.1]{Kho-Ro}). Note that the situation of Lemma~\ref{noncompRees} is precisely that of non-hollow acts. 

\begin{proposition}
An $S$-act $A$ is connected in the category $S$-$\A_0$ if and only if it is hollow.
\end{proposition}
\begin{proof}
Suppose $A$ is hollow and it is not connected, i.e., there is a decomposable factor  $\pi(A) = A_1 \sqcup A_2$ by Proposition~\ref{comp=>indecomp}. Then the preimages $\pi^{-1}(A_1)$ and $\pi^{-1}(A_2)$ form subacts of the act $A$ such that $A=\pi^{-1}(A_1) \cup \pi^{-1}(A_2)$, but neither of $\pi^{-1}(A_i)$ equals $A$. Since the decomposition is proper, we get a contradiction.

On the other hand, if $A$ is not hollow, use the construction of Lemma~\ref{noncompRees}.
\end{proof}

\subsection{Steady monoids}
In accordance with the definition of steady rings (cf. \cite{CT, EGT,Tr}) we say that a monoid (resp. monoid with zero element) $S$  is {\it left steady} (resp. {\it left $0$-steady}) provided every connected left act in the category $S$-$\oA$
(resp. $S$-$\A_0$) is necessarily cyclic.
Note that every cyclic act  is connected by 
Proposition~\ref{cyclic}.

\begin{example}\label{ExZ} \rm 
(1) If $S$ is a group, then it is easy to see that indecomposable $S$-acts are cyclic. Hence connected $S$-acts are precisely cyclic ones by \cite[Theorem I.5.10]{KKM} (cf. Proposition~\ref{comp=>indecomp}, thus groups are (left) steady monoids.

(2) The Pr\"{u}fer group $\mathbb Z_{p^\infty}$ is a connected act over the monoid $(\mathbb N,+,0)$. Clearly, it is not a cyclic $\mathbb N$-act, as it is not a cyclic $\Z$-act. Hence $(\mathbb N,+,0)$ is not steady.
\end{example}

The following assertion presents an analogy of the description of connected projective objects in categories of modules.

\begin{proposition} Let $\mathcal C$ be either $S$-$\oA$ or $S$-$\A_0$. Then a projective left act is connected in $\mathcal C$ if and only if it is cyclic.
\end{proposition}
\begin{proof}
For the direct implication note that, by Theorem \ref{proj-decomp} any projective act has a  decomposition into indecomposable projective subacts, since both $S$-$\oA$ and $S$-$\A_0$ are UD-categories. As it is connected, it is indecomposable by Proposition~\ref{comp=>indecomp}. Now the result follows from Lemma~\ref{generator-retract} since $S$ generates both of the categories $S$-$\oA$ and $S$-$\A_0$.

The reverse implication is a consequence of Proposition \ref{cyclic}.
\end{proof}

A monoid $S$ is called \emph{left perfect} (\emph{left 0-perfect}) if each $A\in S$-$\oA$ ($A\in S$-$\A_0$) has a projective cover, i.e., there exists 
(up to isomorphism unique) a projective $S$-act $P$ and an epimorphism $f: P\to A$ such that for any proper subact $P^{\prime} \subset P$ the restriction 
$f|_{P^{\prime}}: P^{\prime}\to A$ is not an epimorphism (cf. \cite{I, K}).

Analogously to the case of perfect rings, which are known to be steady, we prove that 0-perfect monoids are 0-steady.

\begin{proposition} Let $S$  be a monoid with zero. If $S$ is left $0$-perfect, then connected objects of $S$-$\A_0$ are precisely cyclic acts. 
Hence $S$ is left $0$-steady. 
\end{proposition}
\begin{proof}  Let $A$ be a connected $S$-act and
$\pi\in\Mor(P,A)$ be a projective cover of $A$. Assume that $P$ is not indecomposable with a nontrivial decomposition $(P_0,P_1)$. Then neither $\pi(P_0)$ nor $\pi(P_1)$ is not equal to $A$ and $B=\pi(P_0)\cap \pi(P_1)$ is a subact of $A$. Then $(\pi(P_0)/B, \pi(P_1)/B)$ forms a decomposition of the Rees factor $A/B$. Note that it is non-trivial, otherwise $\pi(P_0)\subseteq \pi(P_1)$ or $\pi(P_1)\subseteq \pi(P_0)$ 
which contradicts to the fact that $\pi(P_0)\ne A\ne \pi(P_1)$.
Since every factor of $A$ is indecomposable by Lemma~\ref{char.oA}, we obtain a contradiction.
\end{proof}

\subsection{Autoconnected acts}

Let us formulate a direct consequence of Lemma~\ref{autoIndec} and Proposition~\ref{endoImage}:

\begin{lemma}\label{autoIndec2} Let $C$ be an a autoconnected object in either $S$-$\A_0$ or $S$-$\oA$ and let $\varphi$ be an endomorphism of $C$. Then
$\varphi(C)$ is autoconnected and indecomposable, in particular, $C$ is indecomposable.
\end{lemma}

Now we can formulate a criterion of autoconnectedness in $S$-$\oA$ (cf. \cite[Lemma 4.1]{Mo10}):

\begin{theorem}\label{ThmAuto} The following conditions are equivalent for an act $C\in S$-$\oA$:
\begin{enumerate}
\item[(1)] $C$ is autoconnected,
\item[(2)] $C$ is connected,
\item[(3)] $C$ is indecomposable.
\end{enumerate}
\end{theorem}

\begin{proof}
The implication (2)$\Rightarrow$(1) is clear, the implication (1)$\Rightarrow$(3) follows from Lemma~\ref{autoIndec2} and the equivalence (2)$\Leftrightarrow$(3)
is proved in Lemma~\ref{char.oA}.
\end{proof}

\begin{example} \rm 
Consider the monoid $\mathcal Z = \left( \mathbb{Z}, \cdot , 1\right) $ and the $\mathcal Z$-act 
$A=2\mathbb Z\cup 3\mathbb Z$ from Examples~\ref{exmplZ} and \ref{Z-act}.
Then $A$ is autoconnected in $S$-$\A_0$, since for any morphism $A \rightarrow \coprod_{i\in I}A_i $ with $A_i \cong A$, the component in which the image lies is determined by the image of the element $6$.
\end{example}

The previous example shows that within the category $S$-$\A_0$ the class of autoconnected acts is in general strictly larger than the class of connected acts; whereas the following example will show that the class of autoconnected acts is in general strictly smaller than that of indecomposable objects, even for left perfect monoids.

\begin{example} \rm
Consider the commutative monoid $\mathcal S = \left( \left\lbrace 0, 1, s, s^2\right\rbrace, \cdot, 1\right)$ (which could be embedded into the multiplicative monoid of the factor ring $\Z[s]/(s^3)$) with the following multiplication table:

{\centering
$
\begin{array}{c||cccc}
\cdot & 0 & 1 & s & s^2 \\ \hline\hline
0 & 0 & 0 & 0 & 0 \\ 
1 & 0 & 1 & s & s^2 \\ 
s & 0 & s & s^2 & 0 \\ 
s^2 & 0 & s^2 & 0 & 0
\end{array}. $

}
\medskip
Then consider the $\mathcal{S}$-act $A = \left\lbrace x,y,z,t, \theta\right\rbrace$ with the action of $\mathcal{S}$ given as follows:
\[
0 \cdot a = \theta , \ \ 1 \cdot a = a \text{ for any } a \in A 
\]
\[
s\cdot x = s \cdot y = z,\ \ 
s \cdot z = t,\ \ 
s \cdot t = \theta.
\]
Then $A$ is indecomposable, while the Rees factor $A/\left\langle z\right\rangle $ decomposes into two isomorphic components (so $A$ is not connected), each of which can be mapped onto $\left\langle t\right\rangle \leq A$, hence $A$ is not autoconnected.

One can furthermore prove that $\mathcal{S}$ is left perfect using \cite[Theorem 1.1]{I}.
\end{example}

For $S$-acts $A_1,A_2\in S$-$\A_0$ denote by $\pi_i:A_1\sqcup A_2\to A_i$, $i=1,2$ the canonical projections and note that any canonical projection is a correctly defined morphism in the category $S$-$\A_0$.

\begin{lemma}\label{autoCompA_0}
Let $C, C_1,C_2\in S$-$\A_0$ and $C\cong C_1\cong C_2$. Then
$C$ is autoconnected if and only if 
for each morphism $f:C\to C_1\sqcup C_2$ 
there exists $i$ such that $\pi_if(C)=\theta$.
\end{lemma}
\begin{proof} The direct implication follows immediately from 
Corollary~\ref{autoCompChar}. 

If $C$ is not autoconnected, then by Corollary~\ref{autoCompChar}(3) there exists a morphism $g:C\to\coprod_{i\in I} C_i$ where $C_i\cong C$ and for each $i\in I$ there exists $j\ne i$ such that  $g(C)\cap \nu_j(C_j)\ne\theta$, which implies that there are two distinct $j_1,j_2\in I$ such that 
$g(C)\cap \nu_{j_k}(C_{j_k})\ne\theta$ for both $k=1,2$. 
Now the composition of $g$ with the canonical projection to $C_i\sqcup C_j$ presents an example of a morphism $f:C\to C_{j_1}\sqcup C_{j_2}$ satisfying 
$\pi_{j_k}f(C)\ne \theta$ for $k=1,2$.
\end{proof}

For a pair $B_1, B_2$ of subacts of a left $S$-act $A$ with inclusions $\iota_i: B_i\to A$ denote by 
$\rho_{B_1B_2}:B_1\sqcup B_2\to A$ the unique morphism satisfying 
$\rho_{B_1B_2}\nu_i=\iota_i$ for $i=1,2$, where $\nu_i$ denotes the coproduct structural morphism.
We finish the paper by a characterization of non-autoconnected $S$-acts in the category $S$-$\A_0$, which can by provided by narrowing the class of non-hollow (i. e. non-connected) acts by

\begin{proposition}\label{PropAuto} The following conditions are equivalent for a triple of isomorphic acts $A,A_1,A_2$ in the category $S$-$\A_0$:
\begin{enumerate}
\item[(1)]$A$ is not autoconnected in $S$-$\A_0$,
\item[(2)] there exists a pair $B_1, B_2$ of proper subacts of $A$ satisfying $A = B_1 \cup B_2$ and there exists a morphism $f: B_1 \sqcup B_2 \to A_1 \sqcup A_2$ such that $\pi_i f(B_1 \sqcup B_2) \neq \theta_{A_i}$ for $i=1,2$ and $ker \rho_{B_1B_2} \subseteq ker \,f$.
\end{enumerate}
\end{proposition}

\begin{proof}
Sufficiency follows from the Homomorphism Theorem \cite[Theorem 4.21]{KKM}, which ensures the existence of a morphism $f^{\prime}: A \to 
A_1\sqcup A_2$: 
\[
\begin{CD}
B_1 \sqcup  B_2 @>f>>   A_1 \sqcup A_2 \\
@V\rho_{B_1B_2} VV     @AAf^{\prime}A \\
A   @= A
\end{CD},
\]
which turns to be the witnessing morphism for non-autoconnectedness thanks to the property $\pi_i f(B_1 \sqcup B_2) \neq \theta_{A_i}$ for both $i=1,2$.

Let $g: A \to A_1 \sqcup A_2$ be the morphism witnessing non-autoconnectedness by Lemma~\ref{autoCompA_0}, hence $\pi_i g (A) \neq \theta_{A_i}$ for both $i = 1,2$. Let $\nu_i: A_i \to A_1 \sqcup A_2$ denote the coproduct structural morphism 
 and set $B_i = g^{-1}\left( g\left( A\right) \cap \nu_i\left( A_i\right) \right) $; then clearly $A = B_1 \cup B_2$. Set now $f= g \rho_{B_1B_2}$.
\end{proof}

{\bf Acknowledgment.} The authors thank the referee for her/his  valuable comments and ideas, which contributed significantly to the improvement of the content of the article.

\end{document}